\documentclass[a4]{amsart}
\usepackage{fullpage}
\usepackage{nicefrac}
\usepackage{enumerate}
\usepackage{amssymb,mathrsfs,mathtools,amsthm,amscd}
\usepackage[english]{babel}
\usepackage[utf8]{inputenc}
\usepackage[T1]{fontenc}

\usepackage{etex,tikz}\usetikzlibrary{patterns}
\usepackage{xcolor}\usepackage[all]{xy}
\def\C{\mathbb{C}} \def\N{\mathbb{N}} \def\P{\mathbb{P}}

\def\B{\mathbb{B}} 
\def\G{\mathbb{G}} \def\A{\mathbb{A}}
\def\X{\mathscr{X}} \def\Y{\mathscr{Y}}

\def\H{\mathscr{H}}
\def\OO{\mathscr{O}}
\def\I{\mathbb{I}}

\def\Q{\mathscr{Q}}
\def\hom{\hat{\omega}}
\def\hX{\hat{X}}

\def\hPhi{\hat{\varPhi}}
\def\w{\mathfrak{w}}
\def\M{\mathfrak{m}}
\def\q{u}
\def\p{v}
\let\leq\leqslant \let\geq\geqslant

\def\bydef{\mathbin{\mathop{:=}}}

\DeclareMathOperator{\Bs}{Bs} 
\DeclareMathOperator{\Pluc}{Pluc} 
\DeclareMathOperator{\Bl}{Bl} 
\DeclareMathOperator{\Gr}{Gr} 
\DeclareMathOperator{\pr}{pr} 
\DeclareMathOperator{\res}{res} 
\DeclareMathOperator{\codim}{codim} 
\DeclareMathOperator{\Mat}{Mat} 
\DeclareMathOperator{\Span}{Span} 
\DeclareMathOperator{\Supp}{Supp} 
\DeclareMathOperator{\rk}{rank} 

\def\midbar{
  \mathchoice
  {\mathrel{\textsl{\Large|}}}
  {\mathrel{\textsl{\large|}}}
  {\mathrel{\textsl{|}}}
  {\mathrel{\textsl{\small|}}}
}
\def\displaymap#1#2{%
  \ifx\relax#1\relax
  \left.\vcenter{\xymatrix@=0pc{#2}}\right.
  \else
  #1\colon\left\vert\vcenter{\xymatrix@=0pc{#2}}\right.
  \fi
}

\newtheorem{theorem}{Theorem}[section]
\newtheorem{lemma}[theorem]{Lemma}
\newtheorem{proposition}[theorem]{Proposition}
\newtheorem{proposition?}[theorem]{??Proposition??}

\newtheorem*{THM}{Main Theorem}

\theoremstyle{definition}
\newtheorem{remark}[theorem]{\it Remark\/}

\def\a{\mathbf{a}}

\author{Damian Brotbek}
\address{Institut de Recherche Mathématique Avancée, Université de Strasbourg}
\email{brotbek@math.unistra.fr}
\title{On the hyperbolicity of general hypersurfaces}
\begin{document}
\maketitle

\begin{abstract}
 In 1970, Kobayashi conjectured that general hypersurfaces of sufficiently large degree in \(\P^n\) are hyperbolic. In this paper we prove that 
 a general sufficiently ample hypersurface in a smooth projective variety is  hyperbolic. To prove this statement, we construct hypersurfaces satisfying a property which is Zariski open and which implies hyperbolicity. These hypersurfaces are chosen such that the geometry of their higher order jet spaces can be related to the geometry of a universal family of complete intersections.  To do so, we  introduce a Wronskian construction which associates a (twisted) jet differential  to every finite family of global sections of a line bundle.
\end{abstract}
\section{Introduction}

A smooth projective variety \(X\) over the field of complex numbers is said to be \emph{Brody hyperbolic} if there  is no non-constant holomorphic map \(f:\C\to X\). In view of a result of Brody \cite{Bro78}, in our situation (when \(X\) is compact), this is equivalent to saying that \(X\) is \emph{Kobayashi hyperbolic}, and we will simply use the word \emph{hyperbolic} in what follows. In   \cite{Kob70,Kob98}, Kobayashi conjectured: \emph{a general hypersurface in \(\P^n\) of sufficiently large degree is hyperbolic}. When \(n=2\), this conjecture  follows from the fact that a curve is hyperbolic if and only if its genus is greater or equal to two.

Before considering the general situation, one might wonder if there exist examples of hyperbolic hypersurfaces in \(\P^n\) with \(n\geq 3\). The first such example in \(\P^3\) was constructed by Brody and Green \cite{B-G77} as hypersurfaces defined by equations of the form
\[X_0^{2r}+X_1^{2r}+X_2^{2r}+X_3^{2r}+aX_0^rX_1^r+bX_0^rX_2^r=0,\]
with  \(r\geq 25\) and general \(a,b\in \C\). Afterwards, many authors have provided examples of this nature,  see for instance \cite{Nad89,ElG96,DEG97,S-Z02} and the work of Masuda and Noguchi \cite{M-N96}, where a considerable amount of examples in any dimension is given. See \cite{Zai03} for more details.

 For the case  \(n= 3\), the first proofs of the Kobayashi conjecture were provided in  \cite{McQ99} (for hypersurfaces of degree \(d\geq 36\)) and \cite{DeG00} (for \(d\geq 21\)), and relied among other things on the  ideas of McQuillan \cite{McQ98} about the entire leaves of foliations on surfaces. The bound was later improved to \(d\geq18\) in \cite{Pau08}.

In a more algebraic direction one can study the positivity of the canonical bundle of subvarieties of general hypersurfaces. Recall that the Green-Griffiths-Lang conjecture \cite{G-G80,Lan86} predicts that varieties of general type are weakly hyperbolic (where we say that a variety \(X\) is \emph{weakly hyperbolic} if all its entire curves lie in a  subvariety \(Z\subsetneq X\)). A positive answer to this conjecture would in particular imply  that a smooth projective variety is hyperbolic if all of its subvarieties are of general type. 
The fact that all subvarieties of (very) general hypersurfaces of large degree in \(\P^n\)  are of general type 
was established by the work of Clemens \cite{Cle86}, Ein \cite{Ein88,Ein91} and Voisin \cite{Voi96}, later improved by Pacienza  \cite{Pac04}.

In \cite{Siu04}, Siu generalized  Voisin's variational method from \cite{Voi96} to higher order jet spaces, and outlined a strategy to  prove Kobayashi's conjecture. 
This motivated a lot of research over the last decade \cite{Rou07,Pau08,Div08,Div09,Mer09,DMR10,D-T10,Dar15,Dar16}, which culminated with the work of Diverio, Merker and Rousseau \cite{DMR10}, and the proof of the weak hyperbolicity of general hypersurfaces in \(\P^n\) of degree \(d\geq 2^{(n-1)^5}\).   Building on \cite{DMR10}, Diverio and Trapani \cite{D-T10} proved that the Kobayashi conjecture holds for (very) general hypersurfaces in \(\P^4\) of degree \(d\geq 593\). The bound of the theorem of \cite{DMR10} was  later improved by different authors \cite{Dem11,Ber15}, the current  best bound being \(d\geq (5n)^2n^n\) \cite{Dar15}. We refer to \cite{Pau14} for more details on this approach. More recently, in \cite{Siu15}, Siu provided more details to the strategy outlined in \cite{Siu04} in order to complete his proof of the Kobayashi conjecture. 

 Lastly, Demailly  developed another approach towards Kobayashi's conjecture \cite{Dem15,Dem15b} based on his work on the Green-Griffiths-Lang conjecture \cite{Dem11,Dem14}. 

In view of the work of Zaidenberg \cite{Zai87},  and on the aforementioned works \cite{Cle86,Ein88,Ein91,Voi96,Pac04}, one can expect a possible bound in the Kobayashi conjecture to be \(d\geq 2n-1\) for \(n\geq 3\).\\

The goal of the present paper is to provide an alternative approach to the Kobayashi  conjecture in order to prove the following statement.\begin{THM}
Let \(X\) be a smooth projective variety. For any  ample line bundle \(A\) on \(X\), there exists \(d_0\in \N\) such that for any \(d\geq d_0\), a general hypersurface \(H\in |A^d|\) is  hyperbolic.
\end{THM}
Note that when \(X=\P^n\) and \(A=\OO_{\P^n}(1)\), this is precisely the Kobayashi conjecture. On the other hand, the Kobayashi conjecture implies  our main result for degree \(d\) sufficiently large and sufficiently divisible. Indeed,  for \(d\) sufficiently large, \(A^d\)  induces an embedding \(X\hookrightarrow \P^N\) such that \(A^d=\OO_{\P^N}(1)|_X\).  Then, taking \(d'\) large enough, the statement for elements in \(|A^{dd'}|\) is reduced to the statement for elements in \(|\OO_{\P^N}(d')|\).

The proof we present here is not effective on \(d_0\) because of two noetherianity arguments. However, shortly after a first version of the present paper was made available on the arXiv, Ya Deng \cite{Den16} was able to render both arguments effective, and  obtained  the bound \(d_0=n^{n+1}(n+1)^{n+2}(n^3+2n^2+2n-1)+n^3+3n^2+3n \leq (n+1)^{2n+6}\) when \(A\) is very ample (where \(n=\dim X\)). 

The main tool of our proof is the use of jet differential equations. Those  can be seen as higher order analogues of symmetric differential forms and provide obstructions to the existence of entire curves \cite{G-G80,S-Y97,Dem97,Dem97b}. 
A fruitful way  to produce jet differential equations on a given variety is to use the Riemann-Roch theorem (see for instance \cite{G-G80,Rou06}) or Demailly's holomorphic Morse inequalities \cite{Dem91} (see for instance \cite{Div08,Div09}, and also \cite{Mer15,Dem11}). However, our proof relies on another construction described below.

A general strategy towards proving a hyperbolicity statement is to construct jet differential equations on the variety under consideration and then to control their base locus  in an adequate jet space. This strategy has already been carried out  successfully as  for instance in  \cite{DMR10} and \cite{Siu15}.

Considering jets of order one, recall that  a conjecture of Debarre \cite{Deb05} predicts that \emph{a general complete intersection in \(\P^n\) of high multidegree and of codimension larger than its dimension, has ample cotangent bundle}.
A natural way to approach this conjecture is to construct symmetric differential forms (jet differential equations of order one) on the complete intersection under consideration, and  to control their base locus. This rises a connection between the Kobayashi and the Debarre conjecture which motivated a conjecture  of  Diverio and Trapani \cite{D-T10}.  
  This connection was investigated   in \cite{Bro14}, where among other things, we used the strategy of \cite{Siu04} and the ideas of \cite{DMR10} to prove the conjecture of Debarre for complete intersection surfaces. 
 Later, in \cite{Bro15}, we proved a higher dimensional result towards this conjecture. To do so, we used the openness property of ampleness to reduce the  statement for general complete intersections of a given multi-degree to the construction of an example. This example was constructed by intersecting (many) particular deformations of Fermat type hypersurfaces, on which we were able to produce explicit  symmetric differential forms.
 Afterwards, in \cite{Xie15}   (see also \cite{Xie16}), Xie was able to prove  the Debarre conjecture (with an explicit bound on the degree) by, among other things, generalizing the symmetric differential forms  constructed in  \cite{Bro15} to a wider class of complete intersections. Independently, in a joint work with Darondeau \cite{B-D15}, we gave a geometric interpretation of the cohomological computations of \cite{Bro15}, in order to give a short proof of the  Debarre conjecture.

In the present paper, we generalize the approach developed in \cite{B-D15} to higher order jet spaces.  To simplify the exposition, we will temporarily restrict  ourselves to the case \(X=\P^n\) and \(A=\OO_{\P^n}(1)\). 
 
 While it is known that hyperbolicity is an open property in the euclidean topology (see \cite{Bro78}), it is unknown whether it is open in the Zariski topology. 
 In order to prove the main theorem, we thus  construct an example of a hypersurface satisfying  a certain ampleness property \eqref{eq:kjethyperbolicity}, which implies hyperbolicity and which is a Zariski open property. The statement for general hypersurfaces will then follow from this particular example. 
 
To construct this example, we use hypersurfaces of the same type as the ones used  in \cite{B-D15}. Consider degree \(d\) homogenous polynomials in \(\C[z_0,\dots, z_n]\) of the form 
\begin{equation}\label{eq:Masuda-Noguchi}
F(\a)=\sum_{\substack{I=(i_0,\dots, i_n)\\ i_0+\cdots+i_n=\delta}}a_Iz^{(r+k)I}
\end{equation}
where we used the multi-index notation \(z^{(r+k)I}\bydef z_0^{(r+k)i_0}\cdots z_n^{(r+k)i_n}\) and where the \(a_I\) are homogenous polynomials of degree \(\varepsilon\), so that \(d=\varepsilon +(r+k)\delta\).

Let us motivate this choice of equations by giving a rough idea about how we use the special form of the polynomials \(F(\a)\)   
to construct jet differential equations on the associated  hypersurfaces. As we will see, the higher order differentials of \(F(\a)\), can be written, locally, as
\begin{equation}\label{eq:differential}\left\{\begin{array}{rcccccc}
F(\a)&=& \sum_{I}a_Iz^{(r+k)I}&=&\sum_I\alpha_I^0z^{rI}&=& \sum_I\alpha_I^0T^I\\
d^{[1]}F(\a)&=& \sum_{I}\tilde{a}_I^1z^{(r+k-1)I}&=&\sum_I\alpha_I^1z^{rI}&=& \sum_I\alpha_I^1T^I\\
&\vdots &&\vdots&&\vdots&\\
d^{[k]}F(\a)&=& \sum_{I}\tilde{a}_I^kz^{(r+k-k)I}&=&\sum_I\alpha_I^kz^{rI}&=& \sum_I\alpha_I^kT^I.
\end{array}\right.\end{equation}
Here,  \(\alpha_I^p=\tilde{a}_I^pz^{(r+k-p)I}\), \(\alpha_I^0=a_Iz^{kI}\) and  \(T_i=z_i^r\) for any \(0\leq i\leq n\), where the \(\tilde{a}_I^p\) should be thought of as differential forms of order \(p\).
 One should think of the equations on the left hand side as the equations defining a suitable \(k\)th order jet space \(H_{\a,k}\) of \(H_{\a}\). Considering \([T_0,\dots, T_n]\) as homogenous coordinates on \(\P^n\), one should think of the equations on the right hand side as the equations of the universal family \(\Y\subset \Gr_{k+1}(H^0(\P^n,\OO_{\P^n}(\delta)))\times \P^n\) of complete intersections of  codimension \(k+1\) and multidegree \((\delta,\dots,\delta)\) in \(\P^n\). Here \(\Gr_{k+1}\) denotes the Grassmannian of \((k+1)\)-dimensional subspaces.  The key point is that a suitable interpretation of \eqref{eq:differential} implies that every element of \(H^0(\Y,q_1^*\Q^m\otimes q_2^*\OO_{\P^n}(-1))\) induces a jet differential equation on \(H_{\a}\) (where \(\Q\) denotes the Pl\"ucker line bundle on the Grassmannian and \(q_1,q_2\) denote the canonical projections). But when  \(k+1\geq n\), the morphism \(q_1\) is generically finite, so that \(q_1^*\Q\) is big and nef. Therefore, for large \(m\), \(H^0(\Y,q_1^*\Q^m\otimes q_2^*\OO_{\P^n}(-1))\) contains many elements, from which  we infer the existence of many jet differential equations on \(H_{\a}\).

The outline of the paper is the following. Section \ref{sse:Demailly-Semple} is devoted to a Wronskian construction which is one of the main tools of this paper. First, the needed properties concerning the Demailly-Semple jet tower are recalled, then the Wronskian associated to families of global sections of a line bundle is defined. This allows us to introduce an ideal sheaf on each stage of the jet tower, whose blow-up satisfies several functorial properties. The aforementioned property \eqref{eq:kjethyperbolicity}, which says that a certain line bundle on such a blow-up is ample, is then introduced. Section \ref{se:SectionProof} is devoted to the proof of our main result. After the  hypersurfaces \(H_\a\) are introduced, the above relationship between the jet space of \(H_{\a}\) and the universal family \(\Y\) is formalized. This is the main technical part of our paper. Once this is done, we explain how the geometry of \(\Y\) is used to prove that the hypersurface \(H_\a\) satisfy \eqref{eq:kjethyperbolicity} and therefore conclude the proof of our main result.\\


In this article, we will work over the field of complex numbers \(\C\). While the objects we consider are mostly of algebraic nature, we work in the analytic category because this is needed on a few occasions. Given a vector bundle \(E\) on a variety \(X\), the projectivization of \emph{lines} in \(E\)  is denoted by \(P(E)\). The tangent bundle of a smooth variety \(X\) is denoted  by \(T_X\) and its cotangent bundle by \(T_X^*\). A property is said to hold for a general member of an algebraic family of projective varieties \(\X\to T\) if it holds for each fiber over a non-empty \emph{Zariski} open subset of \(T\).

\section{Wronskians on the Demailly-Semple jet tower}\label{sse:Demailly-Semple}
In this section we construct the main tool we are going to need in the proof of our main result, namely a suitable type of Wronskians on the Demailly-Semple tower. Wronskians provide  a fundamental tool in the study of entire curve and in particular in Nevanlinna theory (see for instance \cite{NoW14}). A fruitful way to construct Wronskians is to use a connection satisfying some regularity assumption as for instance in \cite{Siu87,Nad89,ElG96,DEG97,Nog11}. By contrast, the Wronskians we introduce in this paper, are associated to  sections of a given line bundle. This approach is certainly more classical, as it is mainly a reinterpretation of the Pl\"ucker coordinates of higher order osculating planes associated to projective curves \cite{Har95}. In the one dimensional case, such objects where already studied for different purposes as for instance in \cite{Gal74,Lak84,Nog97}.

\subsection{The Demailly-Semple jet tower} Let us first recall the results we need from Demailly's foundational  work \cite{Dem97} in which the reader will find all the details of the results  outlined here.  Let \(X\) be an \(n\)-dimensional complex manifold, and denote by \(J_kX\stackrel{p_k}{\to}X\) the \(k\)-th order jet space of \(X\).  This is the set of equivalence classes of holomorphic maps \(\gamma:(\C,0)\to X\) where \(\gamma_1\sim\gamma_2\) if and only if \(\gamma_1^{(p)}(0)=\gamma_2^{(p)}(0)\) for all \(0\leq p\leq k\) (the derivatives being computed in any coordinate chart).  The class of  \(\gamma\) in \(J_kX\) is denoted by \([\gamma]_k\). The map \(p_k\) is defined by \(p_k([\gamma]_k)\bydef \gamma(0)\). The space \(J_kX\) naturally possesses the structure of a \(\C^{nk}\)-fiber bundle. Indeed, any  coordinates \((z_1,\dots, z_n)\) on a chart \(U\subset X\) induce coordinates 
\[\big(z_1,\dots, z_n,z_1',\dots, z_n',\dots, z^{(k)}_1,\dots, z^{(k)}_n\big)\]
on \(p_k^{-1}(U)\), where by definition, a jet \([\gamma]_k\in p_k^{-1}(U)\) has coordinates \(\big(\gamma_1(0),\dots,\gamma_n(0),\dots, \gamma_1^{(k)}(0),\dots, \gamma_n^{(k)}(0)\big).\)

A \emph{directed manifold} is a pair \((X,V)\) where \(X\) is a complex manifold and where \(V\subset T_X\) is a subbundle of \(T_X\). In the present paper we will only need two special cases of this general framework: the \emph{absolute case}, when we consider the directed variety \((X,T_X)\); the \emph{relative case}, when we consider the directed variety \((\X,T_{\X/T})\) where \(\X\to T\) is a smooth projective morphism of quasi-projective varieties. But for the clarity of the exposition, we work in the generality of \cite{Dem97}.

On a directed manifold \((X,V)\) one defines  \(J_kV\stackrel{p_{k}}{\to}X\) to be the subset \(J_kV\subset J_kX\) of all \(k\)-jets of curves \(\gamma:(\C,0)\to X\) tangent to \(V\) (\emph{i.e.} \(\gamma'(t)\in V_{\gamma(t)}\) for all \(t\) in a neighborhood of \(0\)). It can be shown that \(J_kV\) is a subbundle of \(J_kX\).

One denotes by \(\G_k\) the group of germs of \(k\)-jets of biholomorphisms of \((\C,0)\), namely
\[\G_k\bydef\left\{\varphi:t\mapsto a_1t+a_2t^2+\cdots +a_kt^k \midbar a_1\in \C^*\ \text{and} \ a_j\in \C\ \text{for}\ j\geq 2\right\},\]
where composition is taken modulo \(t^{k+1}\).  Given a directed manifold \((X,V)\), the group \(\G_k\) naturally acts on \(J_kV\) by the (right) action \(\varphi\cdot [\gamma]_k\bydef [\gamma\circ \varphi]_k\). For any \(k,m\geq 1\), one can construct a locally free sheaf \(E_{k,m}V^{*}\), the sheaf of \emph{invariant jet differential equations of order \(k\) and degree \(m\)}, satisfying, for any open \(U \subset X\)
\[
E_{k,m}V^*(U)=\left\{Q\in \OO\left(p_k^{-1}(U)\right)\midbar Q(\varphi\cdot [\gamma]_k)=\varphi'(0)^mQ([\gamma]_k)\ \ \forall  [\gamma]_k\in p_k^{-1}(U),\ \forall \varphi\in \G_k\right\}.
\]

In the spirit of \cite{Sem54}, Demailly also constructs for each \(k\geq 1\),  a manifold \(P_kV\) of dimension \(n+k(r-1)\) (where \(r=\rk V\)) equipped with a rank \(r\) vector bundle \(V_k\) satisfying \(P_{k+1}V=P(V_k)\), \(X_0=X\) and \(V_0=V\). 
We will refer to the sequence
\[\cdots \to P_kV\stackrel{\pi_{k}}{\to} P_{k-1}V\stackrel{\pi_{k-1}}{\to} \cdots\to P_1V\stackrel{\pi_{1}}{\to} X_0=X,\]
as the \emph{Demailly-Semple jet tower of \((X,V)\)}. In the absolute case \((X,T_X)\) we will simply write \(X_k\bydef P_kT_X\), and in the relative case \((\X,T_{\X/T})\) we will write \(\X_k^{\rm rel}\bydef P_kT_{\X/T}\). 

For each \(k\geq 1\), \(P_kV\) comes with a tautological line bundle \(\OO_{P_kV}(1)\), and more generally, for any \(a_1,\dots, a_k\in \mathbb{Z}\) we set 
\[\OO_{P_kV}(a_k,\dots, a_1)\bydef \OO_{P_kV}(a_k)\otimes \pi_{k-1,k}^*\OO_{P_{k-1}V}(a_{k-1})\otimes \cdots \otimes \pi_{1,k}^*\OO_{P_1V}(a_1),\] 
where for any \(0\leq p\leq k\), one writes \(\pi_{p,k}\bydef \pi_{p+1}\circ \cdots\circ\pi_{k}\).

From  \cite{Dem97} §5, any germ of curve \(\gamma:(\C,0)\to X\) tangent to \(V\) can be lifted to a germ \(\gamma_{[k]}:(\C,0)\to P_kV\). Moreover, if one denotes by \(J_{k}^{\rm reg}V\bydef\{[\gamma]_k\in J_kV\midbar \gamma'(0)\neq 0\}\) the space of \emph{regular \(k\)-jets} tangent to \(V\), then there exists a morphism \(J_k^{\rm reg}V\to P_kV\), sending \([\gamma]_k\) to \(\gamma_{[k]}(0)\), whose image is an open subset \(P_kV^{\rm reg}\subset P_kV\) which can be identified with the quotient \(J_k^{\rm reg}V/\G_k\) (see Theorem 6.8 in \cite{Dem97}). Let us mention moreover that \(P_kV^{\rm sing}\bydef P_kV\setminus P_kV^{\rm reg}\) is a divisor in \(P_kV\).

From \cite{Dem97} Theorem 6.8,  for any \(k,m\geq 0\) one has
\begin{equation}\label{eq:IsomEkm}
E_{k,m}V^*=(\pi_{0,k})_{*}\OO_{P_kV}(m).
\end{equation}
This isomorphism is described as follows. From Corollary 5.12 in \cite{Dem97}, for any \(w_0\in P_kV\), there exists an open neighborhood \(U_{w_0}\) of \(w_0\) and a family of germs of curves \((\gamma_w)_{w\in U_{w_0}}\), tangent do \(V\) depending holomorphically on \(w\) such that 
\begin{equation}\label{eq:LiftJets}
(\gamma_w)_{[k]}(0)=w\ \ \text{and} \ \ (\gamma_{w})_{[k-1]}'(0)\neq 0,\ \ \ \forall w\in U_{w_0}.
\end{equation}
The image of a  given \(Q\in E_{k,m}V^*(U)\) under the isomorphism \eqref{eq:IsomEkm} is the section \(\sigma\in \OO_{P_k(V)}(m)(\pi_{0,k}^{-1}(U))\) defined by 
\[\sigma(w)=Q([\gamma_{w}]_k)\big((\gamma_w)_{[k-1]}'(0)\big)^{-m}.\]

Every non-constant entire curve \(f:\C\to X\) tangent to \(V\) can be lifted to an entire curve \(f_{[k]}:\C\to P_kV\)  satisfying \(f_{[k]}(t)\in P_kV^{\rm reg}\) if \(f'(t)\neq 0\), so that  
 in particular,  the image of \(f_{[k]}\) isn't entirely contained in  \(P_{k}V^{\rm sing}\).  The following fundamental result shows that the existence of jet differential equations vanishing along some ample divisor provides obstructions to the existence of entire curves. 
\begin{theorem}[Demailly, Green-Griffiths, Siu-Yeung] \label{thm:FundamentalVanishing}
Let \(X\) be a smooth projective variety and \(V\) a subbundle of \(T_X\).  For any non-constant entire curve \(f:\C\to X\) tangent to \(V\), any ample line bundle \(A\) on \(X\), any  \(a_1,\dots, a_k\in \N\) and any \(\omega\in H^0(P_kV,\OO_{P_kV}(a_k,\dots, a_1)\otimes \pi_{0,k}^*A^{-1})\) we have
\[f^{[k]}(\C)\subset (\omega=0).\]
\end{theorem}

\subsection{Wronskians}

\label{sse:Wronskien}We now describe the  Wronskian construction on which we rely in the rest of this paper. Take an \(n\)-dimensional complex  manifold \(X\) and an integer \(k\geq 0\). Let us start with a local construction.  Let \(U\) be an open subset of \(X\),  one can define for every \(0\leq p\leq k\) a \(\C\)-linear map 
\begin{eqnarray*}
d^{[p]}_U:\OO(U)\to \OO\left(p_k^{-1}(U)\right)
\end{eqnarray*}
 by \(d^{[p]}_Uf([\gamma]_k)=(f\circ \gamma)^{(p)}(0)\) for every \(f\in \OO(U)\) and \([\gamma]_k\in p_k^{-1}(U)\subset J_kX\).
One easily verifies that \(d_U^{[p]}f\) is holomorphic and well defined. Indeed, given a  chart in \(U\) with  coordinates \(\underline{z}=(z_1,\dots, z_n)\), by considering the induced coordinates \((\underline{z},\underline{z}',\dots,\underline{z}^{(k)})\), one can describe \(d_U^{[p]}f\) inductively as follows:
\begin{equation}\label{eq:InducDiff}
d_U^{[0]}f=f\ \ \ \text{and}\ \ \ d_U^{[p+1]}f(\underline{z},\dots,\underline{z}^{(k)})=\sum_{m=0}^p\sum_{i=1}^n\frac{\partial d_U^{[p]}f}{\partial z_i^{(m)}}z_i^{(m+1)}\ \ \ \text{for all}\ \ 0\leq p<k,
\end{equation}
from which the holomorphicity follows at once. Observe also that this expression implies that \(d_U^{[p]}f([\gamma]_k)\) only depends only on the jets of \(f\) at \(x\bydef \gamma(0)\) of order less or equal \(p\), by which we mean that it only depends on the class of \(f\) in \(\OO_{X,x}/\mathfrak{m}_{X,x}^{p+1}\),  where \(\mathfrak{m}_{X,x}\) denotes the maximal ideal of \(\OO_{X,x}\). This remark will be used in the proof of Lemma \ref{lem:RegularBpf}.
 We will also need the following generalized Leibniz rule for \(d_U^{[p]}\):
\[d_U^{[p]}(fg)=\sum_{i=0}^p\binom{p}{i}d_U^{[i]}(f)d_U^{[p-i]}(g).\]
Using this differentiation rule, one can construct the Wronskian of any \((k+1)\) holomorphic functions \(f_0,\dots,f_k\in \OO(U)\) by 
\begin{eqnarray}\label{eq:DefWronskien}
W_{U}(f_0,\dots,f_k)\bydef
\left|\begin{array}{ccc}
d_{U}^{[0]}f_0& \cdots &d_{U}^{[0]}f_k\\ 
\vdots  & \ddots & \vdots \\
d_{U}^{[k]}f_0& \cdots &d_{U}^{[k]}f_k
\end{array}\right|
\in
\OO\left(p_k^{-1}(U)\right).
\end{eqnarray}
This object will be most crucial to us. Let us start by proving that this is an invariant jet differential equation. To ease our notation,  in the rest of this paper, for any \(k\in \N\), we set \[k'\bydef \frac{k(k+1)}{2}=1+2+\cdots+k.\] 
\begin{proposition} Same notation as above. For any \(f_0,\dots, f_k\in \OO(U)\), \(W_U(f_0,\dots, f_k)\in E_{k,k'}T_X^*(U)\).
\end{proposition}
\begin{proof} Recall Fa\`a Di Bruno's formula for holomorphic functions \(h,g\) in one variable such that \(h\circ g\) is defined:
\begin{equation}\label{eq:FaaDiBruno}
(h\circ g)^{(p)}(0)=\sum_{i=1}^pP_{p,i}(g)\cdot h^{(i)}( g(0)),
\end{equation}
where \(P_{p,i}(g)\bydef B_{p,i}\left(g'(0),\dots, g^{(p-i+1)}(0)\right)\) and \(B_{p,i}\) denotes a Bell polynomial. One only needs to know that 
\(P_{p,p}(g)=g'(0)^p.\)
 Take \([\gamma]_k\in p_k^{-1}(U)\) and \(\varphi\in \G_k\). For any \(f\in \OO(U)\) and any \(1\leq p\leq k\) one has
\begin{eqnarray*}
d_U^{[p]}f([\gamma\circ \varphi]_k)&=&(f\circ \gamma\circ \varphi)^{(p)}(0)=\sum_{i=1}^pP_{p,i}(\varphi)(f\circ\gamma)^{(i)}(0)= P_{p,p}(\varphi)d_U^{[p]}f([\gamma]_k)+\sum_{i=1}^{p-1}P_{p,i}(\varphi)d_U^{[i]}f([\gamma]_k)\\
&=& \varphi'(0)^pd_U^{[p]}f([\gamma]_k)+\sum_{i=1}^{p-1}P_{p,i}(\varphi)d_U^{[i]}f([\gamma]_k).
\end{eqnarray*}
Applying this formula to \(f=f_{0},\dots,f_{k}\) and by performing elementary operations on the lines in \eqref{eq:DefWronskien} one obtains that 
\[W_U(f_0,\dots,f_k)(\varphi\cdot[\gamma]_k)=\varphi'(0)^{1+2+\cdots+ k}W_U(f_0,\dots, f_k)([\gamma]_k),\]
and therefore, \(W_U(f_0,\dots,f_k)\in E_{k,k'}T_X^*(U)\). 
\end{proof}
We are now going to globalize this construction.  Let  \(L\) be a holomorphic line bundle on \(X\), suppose that \(U\) is  such that \(L|_U\) can be trivialized and fix such a trivialization.   It induces a \(\C\)-linear map \(H^0(X,L)\to \OO(U)\) which to a global section \(s\) associates the element \(s_U\in \OO(U)\) corresponding to \(s\) under our choice of trivialization. By composing this map with \(d^{[p]}_U\), for \(0\leq p\leq k\), one obtains a \(\C\)-linear map
\begin{eqnarray*}
d^{[p]}_U:H^0(X,L)&\to &\OO\left(p_k^{-1}(U)\right)\\
s&\mapsto& d^{[p]}_Us\bydef d^{[p]}_Us_U.
\end{eqnarray*}

Of course, this map depends on our choice of trivialization, and whenever  this map is used, it will be implicitly assumed such a trivialization has been chosen, this should not lead to any confusion.

This allows us to define the Wronskian of global sections \(s_0,\dots, s_k\in H^0(X,L)\) (above \(U\) with respect to our choice of trivialization) by  
\begin{eqnarray*}
W_{U}(s_0,\dots,s_k)\bydef W_{U}(s_{0,U},\dots,s_{k,U}) \in E_{k,k'}T_X^*(U).
\end{eqnarray*}
One  has the following essential property.
\begin{proposition}\label{prop:WronskianGluing} For any \(s_0,\dots, s_k\in H^0(X,L)\), the locally defined jet differential equations \(W_{U}(s_0,\dots, s_k)\) glue together into a section
\[W(s_0,\dots, s_k)\in H^0\left(X,E_{k,k'}T_X^*\otimes L^{k+1}\right).\]
\end{proposition}
\begin{proof}
Consider open subsets \(U_1,U_2\subset X\) on which \(L\) is trivialized and let \(g\in \OO(U_{12})^*\) be the transition function from \(U_2\) to \(U_1\) (with \(U_{12}=U_1\cap U_2\)). By definition, this means that for any \(s\in H^0(X,L)\) 
\[s_{U_1}=g s_{U_2}\in \OO(U_{12}).\]
Applying the generalized Leibniz rule to this relation, one obtains, for each \(0\leq p\leq k\),
\begin{eqnarray*}
d_{U_1}^{[p]}s=d^{[p]}_{U_{12}}s_{U_1}=d^{[p]}_{U_{12}}gs_{U_2}=\sum_{i=0}^p\binom{p}{i}d_{U_{12}}^{[p-i]}gd_{U_{12}}^{[i]}s_{U_2}=gd^{[p]}_{U_2}s+\sum_{i=0}^{p-1}\binom{p}{i}d_{U_{12}}^{[p-i]}gd_{U_{2}}^{[i]}s,
\end{eqnarray*}
where all the functions of this computation are restricted to \(p_k^{-1}(U_{12})\). It suffices then to apply this formula to \(s=s_0,\dots, s_k\) and to perform elementary operations on the lines in \eqref{eq:DefWronskien} to obtain that 
\[W_{U_1}(s_0,\dots,s_k)=g^{k+1}W_{U_2}(s_0,\dots,s_k)\]
over \(p_k^{-1}(U_{12})\), whence the result.
\end{proof}

Observe that by applying the Leibniz rule the same way as in the preceding  proof, one also obtains that if \(A\) is any line bundle on \(X\), then for any \(s_0,\dots,s_k\in H^0(X,L)\) and any \(s\in H^0(X,A)\),
\begin{eqnarray}\label{eq:MultWronskien}
W(s\cdot s_0,\dots,s\cdot s_k)=s^{k+1}W(s_0,\dots,s_k)\in H^0\left(X,E_{k,k'}T_X^*\otimes L^{k+1}\otimes A^{k+1}\right).
\end{eqnarray}

\subsection{The Wronskian ideal sheaf} Take a directed manifold \((X,V)\) where \(X\) is a quasi-projective non-singular variety.
Since \(J_kV\) is a subbundle of \(J_kX\) we obtain, for any \(k,m\in \N\),  a restriction morphism 
\[{\res}_V:E_{k,m}T_X^*\to E_{k,m}V^*.\]
Therefore, for any line bundle \(L\) on \(X\) and any \(s_0,\dots, s_k\in H^0(X,L)\) one obtains a section
\[W^V(s_0,\dots, s_k)\bydef \res_V(W(s_0,\dots,s_k))\in H^0(X,E_{k,k'}V^*\otimes L^{k+1}),\]
and  the corresponding element under isomorphism \eqref{eq:IsomEkm} will be denoted by
\[\omega^V(s_0,\dots, s_k)\in H^0\left(P_kV,\OO_{P_kV}\left(k'\right)\otimes \pi_{0,k}^*L^{k+1}\right).\]
When no confusion can arise we just denote it by \(\omega(s_0,\dots,s_k)\) and in the relative case we will also use the notation \(\omega^{\rm rel}(s_0,\dots, s_k)\). 
Set
\[
\mathbb{W}(P_kV,L)\bydef\Span\left\{\omega(s_0,\dots,s_k)\midbar s_0,\dots,s_k\in H^0(X,L)\right\}\subset  H^0\left(P_kV,\OO_{P_kV}\left(k'\right)\otimes \pi_{0,k}^*L^{k+1}\right),
\]
and define the \emph{\(k\)-th Wronskian ideal sheaf of \(L\)} to be the ideal sheaf defined by \(\mathbb{W}(P_kV,L)\), it is denoted by \(\mathfrak{w}(P_kV,L)\). Recall that this means that if one considers the evaluation map
\[{\rm ev}: \mathbb{W}(P_kV,L)\to \OO_{P_kV}\left(k'\right)\otimes \pi_{0,k}^*L^{k+1}\]
then 
\[\w(P_kV,L)\bydef {\rm im}({\rm ev})\otimes\left(\OO_{P_kV}\left(k'\right)\otimes \pi_{0,k}^*L^{k+1}\right)^{-1}\subset \OO_{P_kV}.\]
Let us first explain that under a strong positivity hypothesis on \(L\) one can control \(\Supp\left(\OO_{P_kV}/\w(P_kV,L)\right)\).

 Recall that one says that \(L\) separates \(k\)-jets at a point \(x\in X\) if the evaluation map
\[H^0(X,L)\to L\otimes \OO_{X,x}/\mathfrak{m}_{X,x}^{k+1}\]
is surjective.  One has the following.
\begin{lemma}\label{lem:RegularBpf}
If \(L\) separates \(k\)-jets at every point of \(X\) then 
\[\Supp\left(\OO_{P_kV}/\w(P_kV,L)\right)\subset P_kV^{\rm sing}.\]
\end{lemma}
\begin{proof}
Since \(P_kV^{\rm reg}\cong J_k^{\rm reg}V/\G_k\subset J_k^{\rm reg}X/\G_k\), in view of \(\eqref{eq:IsomEkm}\), it suffices to show that for any \([\gamma]_k\in J_k^{\rm reg}X\) there exists \(s_0,\dots, s_k\in H^0(X,L)\) such that \(W(s_0,\dots, s_k)([\gamma]_k)\neq 0.\) 

Take a regular \(k\)-jet \([\gamma]_k\in J_k^{\rm reg}X\) and a neighborhood of \(x\bydef \gamma(0)\) with coordinates \((z_1,\dots, z_n)\) centered at \(x\) such that 
\(\gamma(t)=(t,0,\dots, 0)\) for all \(t\) in a neighborhood of \(0\). Take a trivialization of \(L\) over \(U\), and global sections \(s_0,\dots, s_k\in H^0(X,L)\) extending the  elements \(1,z_1,\frac{z_1^2}{2},\dots, \frac{z_1^k}{k!}\in \OO_{X,x}/\mathfrak{m}_{X,x}^{k+1}\cong L\otimes \OO_{X,x}/\mathfrak{m}_{X,x}^{k+1}\).  One immediately checks that 
\(W_U(1,z_1,\dots, \frac{z_1^k}{k!})([\gamma]_k)=1\), hence \(W(s_0,\dots, s_k)([\gamma]_k)\neq 0\).
\end{proof}
We will also need the following statement.
\begin{lemma}\label{lem:Inclusion} If \(L\) is very ample, then for any \(m\geq 0\) one has 
\[\w(P_kV,L^m)\subset \w(P_kV,L^{m+1}).\]
\end{lemma}
\begin{proof}
The assertion is local. Since \(L\) is very ample, \(X\) is covered by open subsets \(U\) of the form \(U=(s\neq 0)\) for \(s\in H^0(X,L)\). Given \(s\in H^0(X,L)\), and \(U=(s\neq 0)\) on obtains from \eqref{eq:MultWronskien} that for any \(s_0,\dots, s_k\in H^0(X,L^m)\),
\[\omega(s\cdot s_0,\dots, s\cdot s_k)=s^{k+1}\omega(s_0,\dots, s_k)\in \mathbb{W}(P_kV,L^{m+1}),\]
where we write \(s^{k+1} \) instead of \(\pi_{0,k}^*s^{k+1}\). The result follows.
\end{proof}
Therefore, given any very ample line bundle \(L\) on \(X\) we have a chain of inclusions
\[
\w\left(P_kV,L\right)\subset \w(P_kV,L^2)\subset\cdots \subset \w(P_kV,L^m)\subset \cdots
\]
By noetherianity, this sequence eventually stabilizes, say after the integer \(m_\infty(P_kV,L)\in \N\) and let us denote the obtained asymptotic ideal sheaf by
\begin{equation}\label{eq:WronskianStab}
\w_{\infty}(P_kV,L)\bydef\w(P_kV,L^{m})\ \ \ \text{for any}\ \ m\geq m_{\infty}(P_kV,L).
\end{equation}
It turns out that this ideal sheaf doesn't depend on the choice of the very ample line bundle \(L\) and is of  purely local nature. To state this result, let us observe that for any \(w\in P_kV\), writing \(x=\pi_{0,k}(w)\), one can define the Wronskian at \(w\) of germs of functions \(f_0,\dots, f_k\in \OO_{X,x}\) by defining
 \[
 \omega_w(f_0,\dots, f_k)=\omega_U(f_0,\dots, f_k)\in \OO_{P_kV,w}
 \]
 where \(U\) is an neighborhood  of \(x\) on which every \(f_i\) is holomorphic, and where the right hand side should be understood as the class, in the local ring, of the Wronskian corresponding to \(W_U(f_0,\dots, f_k)\) under isomorphism \eqref{eq:IsomEkm} and a fixed choice of trivialization of \(\OO_{P_kV}(1)\) in a neigbhorhood of \(w\). With this notation one has the following.
\begin{lemma}\label{lem:WronskienLocal}
Let \(L\) be a very ample line bundle on \(X\). For any \(x\in X\) and any \(w\in P_kV\) such that \(\pi_{0,k}(w)=x\) one has
\[\w_{\infty}(P_kV,L)_w= \big(\omega_w(f_0,\dots,f_k)\big)_{f_0,\dots, f_k\in \OO_{X,x}}\subset \OO_{P_kV,w},\]
where the right hand side denotes the ideal spanned by \(\{\omega_w(f_0,\dots,f_k)\midbar {f_0,\dots, f_k\in \OO_{X,x}}\}\).
\end{lemma}
\begin{proof} That the left hand side is included in the right hand side is obvious. For the other direction, take \(x\in X\), take an open neighborhood of \(x\) with holomorphic coordinates \((z_1,\dots, z_n)\)  and a trivialization for \(L|_U\) such that \(1,z_1,\dots, z_n\in \OO(U)\cong H^0(U,L|_U)\) all extend to global sections of \(H^0(X,L)\). This is possible since \(L\) is very ample. This implies  that for any \(J=(j_1,\dots, j_n)\in \N^n\), \(z^J\bydef z_1^{j_1}\cdots z_n^{j_n}\) extends to a section in \(H^0(X,L^{m})\) for any \(m\geq |J|\), so that in particular for any \(P_0,\dots, P_k\in \C[z_1,\dots,z_n]\), 
\begin{equation}\label{eq:PolyWronsk}
\omega_{w}(P_0,\dots, P_k)\in \w_\infty (P_kV,L)_w.
\end{equation}
Observe that if \(U'\subset U\) is a  neighborhood of \(x\), \(g\in \OO(U')\), \(m\geq 0\), \(1\leq i\leq n\) and  \(0\leq p\leq k\), then there exists \(\tilde{g}\in \OO(p_k^{-1}(U'))\) such that 
\[
d_U^{[p]}(z_i^{m+k}g)=z_i^{m+k-p}\tilde{g}=z_i^m(z_i^{k-p}\tilde{g}),
\]
where we write, by abuse of notation, \(z_i\bydef\pi_{0,k}^*z_i\). In particular,  working at the level of germs, from the definition of \(\omega_w\) and the multilinearity of the determinant, one obtains that for any \(g_0,\dots, g_k\in \OO_{X,x}\) there exists \(q\in \OO_{P_kV,w}\) such that 
\begin{equation}\label{eq:ApproxMadic}
\omega_w(z_i^{m+k}g_0,g_1,\dots, g_k)=z_i^{m}q\in \M^m_{P_kV,w}.
\end{equation}

Take  \(f_0,\dots,f_k\in \OO_{X,x}\). Since for any \(m\geq 0\) and any \(1\leq i \leq n\) one can write \(f_i=P_i+g_i\) with \(P_i\in \C[z_1,\dots,z_n]\) and \(g_i\in \M_{X,x}^{n(m+k)}\), it follows from \eqref{eq:PolyWronsk} and \eqref{eq:ApproxMadic} that 
\[\omega_{w}(f_0,\dots,f_k)\in \w_{\infty}(P_kV,L)_{w}+\M^m_{P_kV,w}.\]
Since this holds for any \(m\in \N\), it follows from Krull's intersection theorem that
\(\omega_{w}(f_0,\dots,f_k)\in \w_{\infty}(P_kV,L)_{w}.\)

\end{proof}
This lemma allows us to define the \emph{asymptotic Wronskian ideal sheaf of \(P_kV\)} by 
\[
 \w_{\infty}(P_kV)\bydef \w_{\infty}(P_kV,L)\subset \OO_{P_kV},
\]
where \(L\) is any very ample line bundle on \(X\).  Moreover, if \(U\subset X\) is an open subset of \(X\) we will also set \(\w_{\infty}(P_kV|_U)\bydef \w_{\infty}(P_kV)|_{\pi_{0,k}^{-1}(U)}\), this is an ideal sheaf on \(\pi_{0,k}^{-1}(U)=P_kV|_U.\)

Lemma \ref{lem:WronskienLocal} also implies that \(\w_\infty\) behaves well under restriction. Namely, for any \((Y,V_Y)\) and \((X,V_X)\) such that \(Y\subset X\) and such that \(V_Y\subset V_X|_Y\), under the induced inclusion \(P_kV_Y\subset P_kV_X\) one has 
\begin{equation}\label{eq:RestrictionWronskien}
\w_{\infty}(P_kV_X)|_{P_kV_Y}=\w_{\infty}(P_kV_Y).
\end{equation}

\subsection{Blow-up of the Wronskian ideal sheaf}\label{sse:Resolution}
The Wronskian sections defined in Section \ref{sse:Wronskien} can certainly not be used as such to apply Theorem \ref{thm:FundamentalVanishing} because of the (positive) twist by \(L^{k+1}\). However they will be the building blocs for the jet differential equations we are going to construct. As a consequence the ideal sheaf \(\w_{\infty}\) will be an obstruction to the positivity result (on a suitable tautological line bundle on the Demailly-Semple jet tower) we aim at. Therefore, we are led to blow up this Wronskian ideal sheaf.

Take a directed manifold \((X,V)\) where \(X\) is a quasi-projective variety (or an euclidian open subset of a quasi-projective variety). With the above notation, define 
\begin{equation}
\hat{P}_kV\bydef \Bl_{\w_{\infty}(P_kV)}(P_kV)\stackrel{\nu_k}{\to} P_kV
\end{equation}
to be to the blow-up of \(P_kV\) along \(\w_\infty(P_kV)\). In the absolute case we will write \(\hX_k=\hat{P}_kT_X\), and in the relative case we will write \(\hat{\X}_k^{\rm rel}=\hat{P}_kT_{\X/T}\). A priori, one doesn't have any control on the singularities of \(\hat{P}_kV\). Let us denote by  \(F\) the effective Cartier divisor on \(\hat{P}_kV\) such that  
\[
\OO_{\hat{P}_kV}(-F)= \nu_k^{-1}\w_{\infty}(P_kV)=\w_{\infty}(P_kV)\cdot \OO_{\hat{P}_kV}.
\]

By the definition of \(\w_{\infty}(P_kV)\) and in view of Lemma \ref{lem:WronskienLocal} one obtains that for any very ample line bundle \(L\) on \(X\), any \(m\geq 0\) and any \(s_0,\dots, s_k\in H^0(X,L^m)\), there exists 
\[
\hom(s_0,\dots, s_k)\in 
H^0\left(
\hat{P}_kV,
\nu_k^{*}\big(\OO_{P_kV}
\left(
k'\right
)
\otimes \pi_{0,k}^*L^{m(k+1)}
\big)\otimes \OO_{\hat{P}_kV}(-F)
\right)
\]
such that, if one denotes by \(F\cdot\) the map induced by the inclusion \(\OO_{\hat{P}_kV}(-F)\to \OO_{\hat{P}_kV}\),
\begin{equation}\label{annulation}
\nu_k^*\omega(s_0,\dots, s_k)=F\cdot \hom(s_0,\dots, s_k).
\end{equation}
Moreover, one obtains that for any \(m\geq m_{\infty}(P_kV,L)\) and any \(\hat{w}\in \hat{P}_kV\) there exists \(s_0,\dots, s_k\in H^0(X,L^m)\) such that 
\begin{equation}\label{eq:TimesF}
\hom(s_0,\dots, s_k)(\hat{w})\neq 0.
\end{equation}
Observe that from \eqref{eq:RestrictionWronskien} one can deduce a functoriality property for these blow-ups. Indeed, for any \((Y,V_Y)\subset (X,V_X)\), the inclusion \(P_kV_Y\subset P_kV_X\) induces an inclusion
\begin{equation}\label{functoriality}
\hat{P}_kV_Y\subset \hat{P}_kV_X.
\end{equation}
Moreover, \(\hat{P}_kV_Y\) is the strict transform of \(P_kV_Y\) in \(\hat{P}_kV_X\) and \(\OO_{\hat{P}_kV_Y}(-F)=\OO_{\hat{P}_kV_X}(-F)|_{\hat{P}_kV_Y}\). An important consequence of Lemma \ref{lem:WronskienLocal} is that this blow-up process behaves well in families.
\begin{proposition}\label{prop:FamilyBlow-up} Let \(\X\stackrel{\rho}{\to} T\) be a smooth and projective morphism between non-singular quasi-projective varieties.  Take \(\nu_k:\hat{\X}_k^{\rm rel}\to \X_k^{\rm rel}\) as above.  For any \(t_0\in T\) writing  \(X_{t_0}\bydef \rho^{-1}(t_0)\), one has 
\[\nu_k^{-1}(X_{t_0,k})=\hX_{t_0,k}\ \ \ \text{and}\ \ \ \OO_{\hat{\X}_k^{\rm rel}}(-F)|_{\nu_k^{-1}(\hX_{t_0,k})}\cong \OO_{\hX_{t_0,k}}(-F).\]
\end{proposition}
\begin{proof}
The key point of the argument is to prove that the family under consideration with the Wronskian  ideal sheaf is locally a product.

Take \(t_0\in T\) and \(x\in X_{t_0}\subset \X\). Take a neighborhood \(U\subset \X\) of \(x\) such that \(U\cong U_1\times U_2\) where \(U_1\subset T\) is a neighborhood of \(t_0\) and where \(U_2\subset \C^n\) and such that under this isomorphism, the map \(\rho\) is identified with the first projection \(p_1:U\to U_1\). This can be achieved since \(\rho\) is a smooth morphism. Denoting by \(p_2:U\to U_2\) the second projection, one obtains an isomorphism
\[\pi_{0,k}^{-1}(U)=P_kT_{U/T}\cong U_1\times P_kT_{U_2}.\]    
Composing it with the second projection, one obtains a morphism
\(p_2^k:P_kT_{U/T}\to P_kT_{U_2}.\)
We are going to prove that 
\begin{equation}\label{eq:GoalResWronsk}
\w_{\infty}(P_kT_{U/T})=(p_2^k)^{-1}\w_{\infty}(P_kT_{U_2}).
\end{equation}
Since \(\w_{\infty}(P_kT_{U/T})=\w_{\infty}(\X^{\rm rel}_k)_{\pi_{0,k}^{-1}(U)}\), this will conclude the proof at once. Indeed, this will imply that 
\[(\pi_{0,k}\circ\nu_k)^{-1}(U)=\hat{P}_kT_{U/T}\cong U_1\times \hat{P}_kT_{U_2},\]
and since moreover, \(\hX_{t_0,k}\cap (\pi_{0,k}\circ \nu_k)^{-1}(U)\cong \hat{P}_kT_{U_2}\) the result will follow.

To prove  \eqref{eq:GoalResWronsk}, we take \(w\in \pi_{0,k}^{-1}(U)=P_kT_{U/T}\)  and   prove the desired equality at the level of stalks at \(w\). Set \(x=\pi_{0,k}(w)\). From Lemma \ref{lem:WronskienLocal}, it follows that \(\w_\infty(P_kT_{U/T})_w\) is spanned by the Wronskians of the form \(\omega^{\rm rel}_w(f_0,\dots, f_k)\) where \(f_0,\dots, f_k\in\OO_{\X,x}\), and that \(\w_{\infty}(P_kT_{U_2})_{p_2^k(w)}\) is spanned by Wronskians of the form \(\omega_{p^k_2(x)}(g_0,\dots, g_k)\) where \(g_0,\dots, g_k\in \OO_{U_2,p_2(x)}\). Observe that for any \(g_0,\dots, g_k\in \OO_{U_2,p_2(x)}\), one has 
\[(p_{2,w}^k)^*\omega_{p^k_2(w)}(g_0,\dots, g_k)=\omega^{\rm rel}_w(p_{2,x}^*g_0,\dots, p_{2,x}^*g_k)\in \OO_{\X_k^{\rm rel},w}.\]
Where \((p_{2,w}^k)^*:\OO_{P_kT_{U_2},p^k_2(w)}\to \OO_{P_kT_{U/T},w}\) and \(p_{2,x}^*:\OO_{U_2,p_2(x)}\to \OO_{U,x}\) are induced by \(p_{2}^k\) and \(p_2\). This proves already that the left hand side of \eqref{eq:GoalResWronsk} contains the right hand side. Take coordinates \((\underline{t})\) centered at \(\rho(x)=p_1(x)\in U_1\) and coordinates \((\underline{z})\) centered at \(p_2(x)\in U_2\). These induce coordinates \((\underline{t},\underline{z})\) on \(U\) centered at \(x\). Observe that for any \(I_0,\dots, I_k\in \N^{\dim T}\) and any \(J_0,\dots, J_k\in \N^n\) one has
\begin{equation}\label{eq:tConstants}
\omega^{\rm rel}_w(t^{I_0}z^{J_0},\dots, t^{I_k}z^{J_k})=t^{I_0+\cdots+I_k}\omega^{\rm rel}_w(z^{J_0},\dots, z^{J_k})=t^{I_0+\cdots+I_k}(p^k_{2,w})^*\omega_{p_2^k(w)}(z^{J_0},\dots, z^{J_k}).
\end{equation}
This follows from the fact that the computation takes place in the relative jet-space, so that one can consider \(t_1,\dots, t_{\dim T}\) as constants, from which the formula follows by multilinearity.  This implies in particular that for any \(P_0,\dots, P_k\in \C[\underline{t},\underline{z}]\), 
\[\omega_w(P_0,\dots, P_k)\in (p_{2,w}^k)^{-1}\big(\w_\infty(P_kT_{U_2})_{p_2^k(w)}\big)= \big((p_{2,w}^k)^{-1}\w_\infty(P_kT_{U_2})\big)_w.\]
From this, \eqref{eq:GoalResWronsk}  follows from Krull's intersection theorem, as in the proof of Lemma \ref{lem:WronskienLocal}.
\end{proof}
\begin{remark} Let us mention that, as was pointed out to us by O. Benoist,  if we take \(\hat{P}_kV\to P_kV\) to be a resolution of the Wronskian ideal sheaf obtained by using a resolution algorithm that commutes with smooth morphisms in the analytic category (as constructed in \cite{Kol07}), then Proposition  \ref{prop:FamilyBlow-up} would still be valid. With this at hand, one could make the rest of the paper with this definition, this wouldn't change anything except that  the proof of \emph{Theorem \ref{thm:nef}\(\Rightarrow\) Main Theorem} in Section \ref{sse:Setting} below would be slightly more involved. While this would allow us to work only with non-singular varieties, we prefer to use the more elementary definition of \(\hat{P}_kV\) above.
\end{remark}
A key point in the proof of the main theorem is the use of a property which is strictly stronger than hyperbolicity and which is Zariski open. This is precisely condition \eqref{eq:kjethyperbolicity} in the following proposition.
\begin{proposition}\label{prop:k-jethyperbolic}
Let \(X\) be a smooth projective variety. If
\begin{equation}\label{eq:kjethyperbolicity}
\exists a_1,\dots,a_k,q\in \N \ \ \text{such that}\ \  \nu_k^*\OO_{X_k}(a_k,\dots,a_1)\otimes \OO_{\hX_k}(-qF)\ \ \text{is ample},\tag{$\ast$}
\end{equation}
then \(X\) is hyperbolic. Moreover, property \eqref{eq:kjethyperbolicity} is a Zariski open property. Namely, given a smooth projective morphism \(\X\stackrel{\rho}{\to}T\) between quasi-projective varieties, if there exists \(t_{0}\in T\) such that \(X_{t_0}\) satisfies \eqref{eq:kjethyperbolicity} then, for general \(t\in T\), \(X_t\) satisfies 
\eqref{eq:kjethyperbolicity}.\end{proposition} 
\begin{proof}
If \eqref{eq:kjethyperbolicity} is satisfied, then one can find integers \(b_1,\dots, b_k,s\in \N\) and an ample line bundle \(A\) on \(X\) such that 
\[\nu_k^*\left(\OO_{X_k}(b_k,\dots,b_1)\otimes \pi_{0,k}^*A^{-1}\right)\otimes \OO_{\hX_k}(-sF)\]
is base point free. From this one sees that multiplication by \(sF\) induces a linear map
\[H^0\left(\hX_k,\nu_k^*\left(\OO_{X_k}(b_k,\dots,b_1)\otimes \pi_{0,k}^*A^{-1}\right)\otimes \OO_{\hX_k}(-sF)\right)\stackrel{\cdot sF}{\to} H^0\left(\hX_k,\nu_k^*\left(\OO_{X_k}(b_k,\dots,b_1)\otimes \pi_{0,k}^*A^{-1}\right)\right),\]
which defines a linear system \(S\bydef{\rm im}(\cdot sF)\) whose base locus \(\Bs(S)\) is included (set theoretically) in \(\Supp(F)\). But this implies, by Lemma \ref{lem:RegularBpf}, that the induced linear system 
\[(\nu_k)_*S\subset H^0\left(X_k,\OO_{X_k}(b_k,\dots,b_1)\otimes\pi_{0,k}^*A^{-1}\right)\]
satisfies \(\Bs((\nu_k)_*S)\subset \Supp\left(\OO_{X_k}/\w_{\infty}(X_k)\right)\subset X_k^{\rm sing}\). Therefore one has in particular that 
\begin{equation}\label{eq:emptyGG}
\Bs\big(\OO_{X_k}(b_k,\dots,b_1)\otimes\pi_{0,k}^*A^{-1}\big)\subset X_k^{\rm sing}.
\end{equation}
Now, if \(f:\C\to X \) is a non-constant entire curve, then Theorem \ref{thm:FundamentalVanishing} (with \(V=T_X\)) implies that 
\[
f_{[k]}(\C)\subset \Bs\big(\OO_{X_k}(b_k,\dots,b_1)\otimes\pi_{0,k}^*A^{-1}\big)\subset X_k^{\rm sing},
\]
which is a impossible since \(f\) is non-constant. From this one deduces that \(X\) is hyperbolic. The second part of the statement, about the Zariski openness,  follows immediately from Proposition \ref{prop:FamilyBlow-up} and the openness property of ampleness.
\end{proof}
\begin{remark} Let us mention that this argument actually proves that condition \eqref{eq:kjethyperbolicity} implies that the Green-Griffiths locus of \(X\), as defined in \cite{D-R13} is empty, this is a direct consequence of \eqref{eq:emptyGG}. In view of Theorem \ref{thm:FundamentalVanishing}, this last condition is well known to imply hyperbolicity (by the  above argument),  and it is in fact a strictly stronger condition, as is explained in \cite{D-R13}.
\end{remark}
\section{Proof of the main theorem}\label{se:SectionProof}

\subsection{Setting}\label{sse:Setting}
Let us introduce the framework in which we will work from now on. Let \(X\) be a smooth \(n\)-dimensional projective variety  and let \(A\) be an ample line bundle on \(X\). Fix integers \(N,k\) such that \(N\geq n\geq 2\) and \(k\geq N-1\). The integer \(N\) should be thought of as the number of ``variables'', and the integer \(k\) as the jet order.

 Let us emphasize that in order to prove the Main Theorem, one could restrict ourselves to the case \(N=n\) and \(k=n-1\). Nevertheless, in view of possible further developments, we work in a slightly greater generality.

 Take \(\nu_k:\hX_k\to X_k\) and \(\OO_{\hX_k}(-F)\) as in Section \ref{sse:Resolution}. Take \(\p_0\in \N\) such that \(A^\p\) is very ample for any \(\p\geq \p_0\). Fix two integers \(\p,\q\geq \p_0\). The reader interested in the case when \(A\) is very ample can take \(\p_0=\p=\q=1\) in the rest of this article.   Let us now fix \(\tau_0,\dots, \tau_N\in H^0(X,A^{\p})\) in general position. 
 Fix also integers \(\varepsilon,\delta,r\geq 1\). Set  \(\mathbb{I}\bydef\{I=(i_0,\dots,i_N)\midbar |I|=\delta\}\). We are going to focus on hypersurfaces of \(X\) defined by sections of the form
\begin{equation}
F(\mathbf{a})\bydef \sum_{I\in \I}a_I\tau^{(r+k)I}\in H^0(X,A^{\q\varepsilon +(r+k)\p\delta}),
\end{equation}
where for all \(I\in \I\), \(a_I\in H^0(X,A^{\q\varepsilon})\), so that \(\mathbf{a}\bydef(a_I)_{I\in \I}\in \A \bydef \bigoplus_{I\in \I}H^0(X,A^{\q\varepsilon})\). Here we used the multi-index notation \(\tau^I=\tau_0^{i_0}\cdots\tau_N^{i_N}\) for \(I=(i_0,\dots, i_N)\).
Consider the universal family 
\[\H\bydef\left\{(\a,x)\in \A\times X\midbar F(\a)(x)=0\right\}.
\]
Let us denote by \(\rho:\H\to \A\) the natural projection. For any \(\a\in \A\), set \(H_{\a}\bydef \rho^{-1}(\a)\), and let us consider the smooth locus
\(\A_{\rm sm}\bydef \{\a\in \A\midbar H_{\a} \ \text{is smooth}\}\) which is a non-empty Zariski open subset of \(\A\). Let us also denote by \(\rho:\H\to \A_{\rm sm}\) the  restricted family. One has inclusions \(\H_{k}^{\rm rel}\subset \A_{\rm sm}\times X_k\) and from \eqref{functoriality} one has an inclusion \(\hat{\H}_k^{\rm rel}\subset \A_{\rm sm}\times \hX_k\). Denoting by \(\hat{\rho}_{k}:\hat{\H}^{\rm rel}_k\to \A_{\rm sm}\) the natural projection, in view of Proposition \ref{prop:FamilyBlow-up}, one obtains that for any \(\a\in \A_{\rm sm}\), \(\hat{H}_{k,\a}\bydef \hat{\rho}_{k}^{-1}(\a)\cong \hat{H}_{\a,k}\subset \hX_k\).

With the notation of \eqref{eq:WronskianStab}, let us set \(m_{\infty}\bydef m_\infty(X_k,A^{\q})\). The aim of the rest of this paper is to prove the following result.
\begin{theorem}\label{thm:nef} Take \(\p,\q\geq \p_0\). Suppose \(N\geq n\), \(k\geq N-1\), \(\varepsilon \geq m_{\infty}\) and \(\delta\geq n(k+1)\). There exists \(M= M(N,k,\delta)\in \N\),  and \(r(\p,\q,M,N,k,\varepsilon,\delta)\in \N\)  such that if \(r\geq r(\p,\q,M,N,k,\varepsilon, \delta)\), then
there exists a non-empty Zariski open subset \(\A_{\rm nef}\subset \A_{\rm sm}\) such that for any \(\a\in \A_{\rm nef}\) the line bundle
\[\nu_{k}^*\Big(\OO_{X_k}\big(Mk'\big)\otimes \pi_{0,k}^*A^{-1}\Big)\otimes \OO_{\hat{X}_k}(-MF)|_{\hat{H}_{k,\a}}\]
is nef on \(\hat{H}_{k,\a}\).
\end{theorem}
Let us first explain how this theorem implies our main result.
\begin{proof}[Theorem \ref{thm:nef}\(\Rightarrow\) Main Theorem] 
Since \(H_{\a,k}\cong H_{k,\a}\subset \hX_k\) and that   \(\OO_{\hat{X}_k}(-MF)|_{\hat{H}_{k,\a}}\cong \OO_{\hat{H}_{\a,k}}(-MF)\), the conclusion of the theorem implies, after tensoring by suitable line bundles, that \(H_{\a}\) satisfies property \eqref{eq:kjethyperbolicity} for any \(\a\in \A_{\rm nef}\). By Proposition \ref{prop:k-jethyperbolic}, one deduces  that for \(\p,\q,\varepsilon,\delta, r\)  as above, general hypersurfaces in \(|A^{\q\varepsilon+(r+k)\p\delta}|\) satisfy property \eqref{eq:kjethyperbolicity} and are therefore hyperbolic.

 To conclude the proof, it suffices to show, by adjusting the different exponents, that this gives the result for general hypersurfaces in \(|A^{d}|\) for all \(d\) large enough. This can be seen as follows. Take \(\delta=n(k+1)\), \(\p=\p_0\) and \(\q\geq \p_0\) such that \(\gcd (\q,\p\delta)=1\). Take \(R\bydef\max\left\{r(\p,\q,M,N,k,\varepsilon,\delta)\midbar m_{\infty}\leq \varepsilon<m_{\infty}+\p\delta\right\}\) and set \(d_0\bydef \q(m_{\infty}+\p\delta)+(R+k)\p\delta\). We will show the result holds for any \(d\geq d_0\).

It suffices to prove that any integer \(d\geq d_0\) can be written as \(d=\q\varepsilon+(r+k)\p\delta\) for \(r\geq R\) and \(m_{\infty}\leq \varepsilon< m_{\infty}+\p\delta\).  For \(d\geq d_0\), take \(\varepsilon\) to be the unique element in \(\{m_{\infty},\dots, m_{\infty}+\p\delta-1\}\) such that \(\q\varepsilon \equiv d \ [\p\delta]\), which is possible since \(\gcd(\q,\p\delta)=1\). Then \(d-\q\varepsilon=t \p\delta\) for some \(t\in \mathbb{Z}\). But since \(d\geq d_0\), one has \(t \p\delta\geq (R+k)\p\delta\), and it suffices to take \(r=t-k\geq R\) to conclude the proof. 
\end{proof}
\subsection{Maps to the Grassmanian} In this entire section,  take \(N\geq n\geq 2\) and \(k\geq 1\). Note that the hypothesis on  \(k\) is less restrictive  than the hypothesis of Theorem \ref{thm:nef}, while this is useless for the proof of that theorem, we do this in order to present the results of this section in there correct generality.  The main idea in the proof of the positivity statement in Theorem \ref{thm:nef} is to construct a  map from \(\hat{\H}_k^{\rm rel}\) to a suitable generically finite family and to use the positivity of the tautological bundle on the parameter space of this family. Before doing so, we need some preliminaries which we describe in this section.
Let us start with several computational lemmata. 
\begin{lemma}\label{lem:localfacto}  Let \(U\) be an open subset of \(X\) on which \(A\) can be trivialized, and fix such a trivialization. Take \(I=(i_0,\dots,i_N)\). For any \(0\leq p\leq k\) there exists a \(\C\)-linear map
\[d^{[p]}_{I,U}:H^0(X,A^{\q\varepsilon})\to \OO\left(p_k^{-1}(U)\right)\]
such that for any \(a\in H^0(X,A^{\q\varepsilon})\),
\(d_{U}^{[p]}(a\tau^{(r+k)I})=\tau_U^{rI}d^{[p]}_{I,U}(a).\)
\end{lemma}
\begin{proof}
By induction, there exists \(\tilde{a}\) such that \(d^{[p]}_U(a\tau^{(r+k)I})=\tau_U^{(r+k-p)I}\tilde{a}\), it suffices then to define \(d^{[p]}_{I,U}(a)\bydef \tau_U^{(k-p)I}\tilde{a}\).
\end{proof}

Therefore, given any open subset \(U\) any trivialization of \(A|_U\) as in Lemma \ref{lem:localfacto},  any  \(I_0,\dots, I_k\in \I\) and any \(a_{I_0},\dots, a_{I_k}\in H^0(X,A^{\q\varepsilon})\) one can define 
\begin{eqnarray}
W_{U,I_0,\dots, I_k}(a_{I_0},\dots,a_{I_k})\bydef
\left|\begin{array}{cccc}
d_{I_0,U}^{[0]}(a_{I_0}) & \cdots &d_{I_k,U}^{[0]}(a_{I_k})\\ 
\vdots & \ddots & \vdots \\
d_{I_0,U}^{[k]}(a_{I_0})& \cdots &d_{I_k,U}^{[k]}(a_{I_k})\\ 
\end{array}\right|
\in
\OO(p_k^{-1}(U)).
\end{eqnarray}
From Lemma \ref{lem:localfacto} one deduces at once that 
\[W_{U}\big(a_{I_0}\mathbf{\tau}^{(r+k)I_0},\dots,a_{I_k}\mathbf{\tau}^{(r+k)I_k}\big) =\tau_U^{r(I_0+\dots+I_k)}W_{{U},I_0,\dots, I_k}(a_{I_0},\dots,a_{I_k}).\]
Therefore from Proposition \ref{prop:WronskianGluing} one deduces the following.
\begin{lemma}\label{lem:homgluing}
For any  \(I_0,\dots, I_k\in \I\) and any \(a_{I_0},\dots, a_{I_k}\in H^0(X,A^{\q\varepsilon})\), the locally defined functions \(W_{{U},I_0,\dots, I_k}(a_{I_0},\dots,a_{I_k})\) glue together into a global section
\[
W_{I_0,\dots, I_k}(a_{I_0},\dots,a_{I_k})\in H^0\left(X,E_{k,k'}T_X^*\otimes A^{(k+1)(\q\varepsilon+k\p\delta)}\right),
\]
such that 
\(W\big(a_{I_0}\mathbf{\tau}^{(r+k)I_0},\dots,a_{I_k}\mathbf{\tau}^{(r+k)I_k}\big) =\tau^{r(I_0+\dots+I_k)}W_{I_0,\dots, I_k}(a_{I_0},\dots,a_{I_k}).\)
\end{lemma}
Let us denote the global section induced via isomorphism \eqref{eq:IsomEkm} by 
\begin{equation*}\label{eq:DefOmegaHat}
\omega_{I_0,\dots, I_k}(a_{I_0},\dots,a_{I_k})\in H^0\big(X_k,\OO_{X_k}(k')\otimes \pi_{0,k}^*A^{(k+1)(\q\varepsilon+k\p\delta)}\big).
\end{equation*}
Note that  the line bundle involved doesn't depend on \(r\). With this, consider the rational map
\begin{eqnarray*}
\varPhi:\A\times X_k&\dashrightarrow&P\left(\Lambda^{k+1}\C^{\I}\right)\\
(\a,w)&\mapsto&\big(\left[\omega_{I_0,\dots,I_k}(a_{I_0},\dots,a_{I_k})(w)\right]\big)_{I_0,\dots,I_k\in \I},\nonumber
\end{eqnarray*}
where \(\C^{\I}\bydef\bigoplus_{I\in \I}\C\cong \C^{\binom{N+\delta}{\delta}}\). One can see that \(\varPhi\) factors through the Pl\"{u}cker embedding. Indeed, given \(w_0\in X_k\), take \(U_{w_0}\) and \((\gamma_w)_{w\in  U_{w_0}}\) as in \eqref{eq:LiftJets} and \(U\) as in Lemma \ref{lem:localfacto} such that \(U_{w_0}\subset \pi_{0,k}^{-1}(U)\) . For any \(\a=(a_I)_{I\in \I}\in \A\), any \(w\in U_{w_0}\) and any \(0\leq p\leq k\) let us denote by 
\[d^{[p]}_{\bullet,w_0}(\a,w)\bydef \left(d^{[p]}_{I,U}(a_I)([\gamma_w]_k)\right)_{I\in\I}\in\C^{\I} .\]
 This definition depends on the choice of \(w_0\), the choice of the family \((\gamma_w)\) and the choice of the trivialization of \(A\) over  \(U\). Nevertheless, one can  consider the rational map
\begin{eqnarray}
\varPhi_{w_0}:\A\times U_{w_0}&\dashrightarrow&\Gr_{k+1}\left(\C^{\I}\right)\label{eq:DescriptionPhi}\\ 
(\a,w)&\mapsto& \Span\left(d_{\bullet,w_0}^{[0]}(\a,w),\dots, d_{\bullet,w_0}^{[k]}(\a,w)\right).\nonumber
\end{eqnarray}
And one easily observes that if  \(\Pluc:\Gr_{k+1}\left(\C^{\I}\right)\hookrightarrow P\left(\Lambda^{k+1}\C^{\I}\right)\) denotes the Pl\"ucker embedding, then   one has \(\varPhi|_{U_{w_0}}=\Pluc \circ \varPhi_{w_0}\).  This proves that \(\varPhi\) factors through \(\Pluc\) and we will denote (slightly abusively) by \(\varPhi: \A\times X_k\dashrightarrow \Gr_{k+1}(\C^{\I})\) the induced map into the Grassmaniann. 

Our aim is  to prove that \(\nu_k\) partially resolves the singularities of \(\varPhi\). Recall that for any \(I_0,\dots, I_k\in \I\) and any \(a_{I_0},\dots, a_{I_k}\in H^0(X,A^{\q\varepsilon})\) one has
\begin{equation*}
\tau^{r(I_0+\cdots+I_k)}\omega_{I_0,\dots, I_k}(a_{I_0},\dots,a_{I_k})\in H^0\left(X_k,\OO_{X_k}(k')\otimes \pi_{0,k}^*A^{(k+1)(\q\varepsilon+(k+r)\p\delta)}\otimes \w_{\infty}(X_k)\right).\end{equation*}
But from Lemma \ref{lem:WronskienLocal}, we obtain that \(\tau^{r(I_0+\cdots+I_k)}\) doesn't vanish along any irreducible or embedded component of the scheme defined by \(\w_\infty(X_k)\). Form this one deduces that \(\omega_{I_0,\dots, I_k}(a_{I_0},\dots,a_{I_k})\) vanishes along \(\w_\infty(X_k)\),
 which implies the existence of a global section
\[
\hom_{I_0,\dots,I_k}(a_{I_0},\dots,{a_{I_k}})\in H^0\left(\hX_k,\nu_k^*\big(\OO_{X_k}(k')\otimes \pi_{0,k}^*A^{(k+1)(\q\varepsilon+k\p\delta)}\big)\otimes \OO_{\hX_k}(-F)\right)
\]
such that 
\[\nu_k^*\omega_{I_0,\dots,I_k}(a_{I_0},\dots,{a_{I_k}})=F\cdot \hom_{I_0,\dots,I_k}(a_{I_0},\dots,{a_{I_k}}).\]
From the multilinearity property of \(\hom_{I_0,\dots,I_k}(a_{I_0},\dots, a_{I_k})\) it makes sense to consider the rational map
\begin{eqnarray*}
\hPhi:\A\times \hX_k&\dashrightarrow& P\left(\Lambda^{k+1}\C^{\I}\right)\\
(\a,\hat{w})&\mapsto&\left[\big(\hom_{I_0,\dots,I_k}(a_{I_0},\dots,a_{I_k})(\hat{w})\big)_{I_0,\dots,I_k\in \I}\right].
\end{eqnarray*}
Observe that outside \(\Supp(F)\) one has \(\hPhi=\varPhi\circ \nu_k\), therefore, since \(\hX_k\) is irreducible, \(\hPhi\) also factors through the Pl\"ucker embedding, and denote also by \(\hPhi\) the obtained map
\[\hPhi:\A\times \hX_k\dashrightarrow \Gr_{k+1}\left(\C^{\I}\right).\]
We will need a local description for \(\hPhi\) similar to \eqref{eq:DescriptionPhi}.

\begin{lemma}\label{lem:LemmeFonda}
Suppose \(\varepsilon \geq m_{\infty}\). For any \(\hat{w}_0\in \hX_k\) there exists a open neighborhood \(\hat{U}_{\hat{w}_0}\subset \hX_k\) of \(\hat{w}_0\) satisfying the following. For any \(I\in \I\) and any \(0\leq p\leq k\), there exists a linear map
\[\ell^p_{I}:H^0(X,A^{\q\varepsilon})\to \OO(\hat{U}_{\hat{w}_0})\]
such that for any \((\a,\hat{w})\in \A\times \hat{U}_{\hat{w}_0}\), writting \(\ell_{\bullet}^p(\a,\hat{w})=\left(\ell^p_{I}(a_I)(\hat{w})\right)_{I\in \I}\in \C^{\I}\) one has: 
\begin{enumerate}
\item The Pl\"ucker coordinates of \(\hat\varPhi(\a,\hat{w})\) are all vanishing if and only if
 \[\dim\Span\left(\ell_{\bullet}^0(\a,\hat{w}),\dots, \ell_{\bullet}^k(\a,\hat{w})\right)<k+1.\]
\item If \(\dim\Span\left(\ell_{\bullet}^0(\a,\hat{w}),\dots, \ell_{\bullet}^k(\a,\hat{w})\right)=k+1,\) then 
 \[\hPhi(\a,\hat{w})=\Span\left(\ell_{\bullet}^0(\a,\hat{w}),\dots, \ell_{\bullet}^k(\a,\hat{w})\right)\in \Gr_{k+1}\left(\C^\I\right). 
\]
\end{enumerate}

\end{lemma}
\begin{proof} 
From \eqref{eq:TimesF} one knows that there exists \(\tilde{b}_0,\dots, \tilde{b}_k\in H^0(X,A^{\q\varepsilon})\) such that
\begin{equation}\label{eq:DefBaseb}
\hom(\tilde{b}_0,\dots, \tilde{b}_k)(\hat{w}_0)\neq 0.
\end{equation}
Set \(w_0=\nu_k(\hat{w}_0)\) and \(x=\pi_{0,k}(w_0)\). Take \(\tilde{s}\in H^0(X,A^{\p})\) such that \(\tilde{s}(x)\neq 0\), take an open neighborhood  \(U\subset (\tilde{s}\neq 0)\) of \(x\) and a trivialization of \(A|_U\). Set \(s=\tilde{s}^{(r+k)\delta}\in H^0(X,A^{(r+k)\p\delta})\) and  define \(b_0=s\tilde{b}_0,\dots, b_k=s\tilde{b}_k\in H^0(X,A^{\q\varepsilon+(r+k)\p\delta})\). Moreover, take a neighborhood \(U_{w_0}\subset X_k\) of \(w_0\) and a family \((\gamma_{w})_{w\in U_{w_0}}\) as in \eqref{eq:LiftJets}, we can suppose \(\pi_{0,k}(U_{w_0})\subset U\). Take a neighborhood \(\hat{U}_{\hat{w}_0}\) of \(\hat{w}_0\) on which \(\hom(\tilde{b}_0,\dots,\tilde{b}_k)\) never vanishes, and such that \(\nu_k(\hat{U}_{\hat{w}_0})\subset U_{w_0}\). For any \(m\geq 0\)  any \(\sigma\in H^0(X,A^m)\) any \(0\leq p\leq k\) and any \(\hat{w}\in \hat{U}_{\hat{w}_0}\) define \(d^{[p]}_U\sigma(\hat{w})\bydef d^{[p]}_U\sigma ([\gamma_{\nu_k(\hat{w})}]_k)\). This defines an element \(d^{[p]}_U\sigma\in \OO(\hat{U}_{\hat{w}_0})\) and similarly, define for each \(I\in \I\), an element \(d^{[p]}_{I,U}\sigma\in \OO(\hat{U}_{\hat{w}_0})\) for any \(\sigma\in H^0(X,A^{\q \varepsilon})\). 
Let us fix the trivialization of \(\OO_{X_k}(k')|_{U_{w_0}}\) induced by \((\gamma_{w}'(0))_{w\in U_{w_0}}\in \Gamma(U_{w_0},\OO_{X_k}(-1))\). Let us also fix a local generator \(F_{\hat{U}_{\hat{w}_0}}\in \OO(\hat{U}_{\hat{w}_0})\) of the Cartier divisor \(F\).

In this setting, consider the matrix 
\[
G\bydef\left(
\begin{array}{ccc}
d_U^{[0]}(b_0)&\cdots&d_U^{[0]}(b_k)\\
\vdots&&\vdots\\
d_U^{[k]}(b_0)&\cdots&d_U^{[k]}(b_k)
\end{array}
\right)\in \Mat_{k+1,k+1}\big(\OO(\hat{U}_{\hat{w}_0})\big),
\]
And define, for any \(I\in \I\) linear maps \(\ell_I^0,\dots, \ell_I^k:H^0(X,A^{\q\varepsilon})\to \OO(\hat{U}_{\hat{w}_0})\) by 
\begin{equation}\label{eq:Defhatell}
\left(
\begin{array}{c}
\ell_I^0(a_I)\\
\vdots\\
\ell_I^k(a_I)
\end{array}
\right)
=G^{-1}
\left(
\begin{array}{c}
d^{[0]}_{I,U}(a_I)\\
\vdots\\
d_{I,U}^{[k]}(a_I)
\end{array}
\right)=
\frac{1}{\tau_U^{rI}}
G^{-1}
\left(
\begin{array}{c}
d^{[0]}_{U}(a_I\tau^{(r+k)I})\\
\vdots\\
d_{U}^{[k]}(a_I\tau^{(r+k)I})
\end{array}
\right)
\in  \Mat_{k+1,1}\big(\OO(\hat{U}_{\hat{w}_0})\big).
\end{equation}
The key point is to see that this is well defined, namely that for any \(0\leq p\leq k\), \(\ell^p_I(a_I)\in \OO(\hat{U}_{\hat{w}_0})\). To see this observe that, as in the construction of \(\hom_{I_0,\dots, I_k}\),  one obtains
\[\omega(b_0,\dots,b_{p-1},a_I\tau^{(r+k)I},b_{p+1},\dots, b_k)_{U_{w_0}}=\tau_U^{rI}\omega_{p,I}(b_0,\dots,b_{p-1},a_I,b_{p+1},\dots, b_k)\] 
for some \(\omega_{p,I}(b_0,\dots,a_I,\dots, b_k)\in \Gamma(U_{w_0},\w_\infty(X_k))\). Therefore, one can write 
\[\nu_k^*\omega_{p,I}(b_0,\dots,a_I,\dots, b_k)=F_{\hat{U}_{\hat{w}_0}} \hom_{p,I}(b_0,\dots,a_I,\dots, b_k),\]
for some \(\hom_{p,I}(b_0,\dots,a_I,\dots, b_k)\in \OO(\hat{U}_{\hat{w}_0})\).  For each \(0\leq p\leq k\), applying Cramer's rule, one obtains  from the definition of \(\omega\) and \(\hom\) that
\begin{eqnarray*}
\ell^p_I(a_I)&=&\frac{1}{\tau_U^{rI}\det G}
\left|
\begin{array}{ccccccc}
d^{[0]}_{U}(b_0)&\cdots &d^{[0]}_U(b_{p-1})&d_U^{[0]}(a_I\tau^{(r+k)I})&d^{[0]}_U(b_{p+1})&\cdots& d^{[0]}_{U}(b_k)\\
\vdots &&\vdots&\vdots&\vdots&&\vdots\\
d^{[k]}_{U}(b_0)&\cdots & d^{[k]}_U(b_{p-1})&d_U^{[k]}(a_I\tau^{(r+k)I})&d^{[k]}_U(b_{p+1})&\cdots& d^{[k]}_{U}(b_k)
\end{array}
\right|\\
&=&\frac{\nu_k^*\omega(b_0,\dots,b_{p-1},a_I\tau^{(r+k)I},b_{p+1},\dots, b_k)}{\tau^{rI}_U\nu_k^*\omega(b_0,\dots,b_k)_{U_{w_0}}}=\frac{\nu_k^*\omega_{p,I}(b_0,\dots,b_{p-1},a_I,b_{p+1},\dots, b_k)}{\nu_k^*\omega(b_0,\dots,b_k)_{U_{w_0}}}\\
&=&\frac{F_{\hat{U}_{\hat{w}_0}}\hom_{p,I}(b_0,\dots,a_I,\dots, b_k)}{F_{\hat{U}_{\hat{w}_0}}\hom(b_0,\dots,b_k)_{\hat{U}_{\hat{w}_0}}}
=\frac{\hom_{p,I}(b_0,\dots,a_I,\dots, b_k)}{\hom(b_0,\dots,b_k)_{\hat{U}_{\hat{w}_0}}},
\end{eqnarray*}
where we used \eqref{annulation}. Since from \eqref{eq:DefBaseb} and \eqref{eq:MultWronskien} it follows that \(\hom(b_0,\dots, b_k)\) never vanishes on \(\hat{U}_{\hat{w}_0}\), from which the desired holomorphicity follows. 
 
 With the notation of the statement of the lemma, a straightforward computation shows that  the Pl\"ucker coordinates  of 
\(
\Span\left(\ell_{\bullet}^0(\a,\hat{w}),\dots, \ell_{\bullet}^k(\a,\hat{w})\right)
\)
are given by 
\begin{equation}\label{eq:Plucker}
\big(\hom_{I_0,\dots,I_k}(a_{I_0},\dots, a_{I_k})(\hat{w})\big)_{I_0,\dots,I_k\in \I}\in \Lambda^{k+1}\C^\I\ \ {\rm mod} \ \C^*.
\end{equation}
Indeed, for any \(I_0,\dots, I_k\in \I\),
\begin{eqnarray*}
\left|
\begin{array}{ccc}
\ell_{I_0}^0(a_{I_0})&\cdots&\ell_{I_k}^0(a_{I_k})\\
\vdots&&\vdots\\
\ell_{I_0}^k(a_{I_0})&\cdots&\ell_{I_k}^k(a_{I_k})
\end{array}
\right|(\hat{w})
&=&
\frac{1}{\det G}
\left|
\begin{array}{ccc}
d^{[0]}_{I_0,U}(a_{I_0})&\cdots & d^{[0]}_{I_k,U}(a_{I_k})\\
\vdots &&\vdots\\
d^{[k]}_{I_0,U}(a_{I_0})&\cdots & d^{[k]}_{I_k,U}(a_{I_k})
\end{array}
\right|(\hat{w})\\
&=&\frac{\nu_k^*\omega_{I_0,\dots,I_k}(a_{I_0},\dots, a_{I_k})_{U_{w_0}}(\hat{w})}{\nu_k^*\omega(b_0,\dots, b_k)_{U_{w_0}}(\hat{w})}=\frac{\hom_{I_0,\dots,I_k}(a_{I_0},\dots, a_{I_k})_{\hat{U}_{\hat{w}_0}}(\hat{w})}{\hom(b_0,\dots, b_k)_{\hat{U}_{\hat{w}_0}}(\hat{w})}.
\end{eqnarray*}
Since \(\hom_{\hat{U}_k}{(b_0,\dots, b_k})(\hat{w})\) is independent of \(I_0,\dots, I_k\), this proves \eqref{eq:Plucker}, and from this, both statements follow at once.
\end{proof}
Before continuing, we need to introduce some  notation. For any \(x\in X\), define
\[N_x\bydef \#\big\{j\in \{0,\dots,N\}\midbar \tau_j(x)\neq0\big\}\ \ \ \text{and}\ \ \  \I_x\bydef\big\{I\in\I\midbar \tau^{I}(x)\neq 0\big\}.\]
 Observe that since the \(\tau_j\)'s are in general position, and since \(N\geq n\), one has \(N_x\geq 1\) for all \(x\in X\). Let us also define
 \[\Sigma \bydef  \{x\in X\midbar N_x=1\}\ \ \ \text{and}\ \ \ X^\circ\bydef X\setminus \Sigma=\{x\in X\midbar N_x\geq2\}.\]
 If \(N>n\) then \(N_x\geq 2\) for all \(x\in X\), therefore \(\Sigma=\varnothing\) and  \(X^\circ=X\). If \(N=n\) then \(\dim\Sigma=0\). Observe moreover that 
 \begin{equation}\label{eq:BorneIx}
\#\I_x= \binom{N_x-1+\delta}{\delta}\ \ \text{for all}\ \  x\in X \ \ \ \text{and therefore}\ \ \  \#\I_x\geq \delta+1 \ \ \text{for all}\ \ x\in X^\circ.
\end{equation}
 For any \(x\in X\),  write \(\C^{\I_x}=\bigoplus_{I\in \I_x}\C\), one obtains a natural projection map 
\(\rho_x:\big(\C^{\I}\big)^{k+1}\to \big(\C^{\I_x}\big)^{k+1}.\)

We will from now on suppose that \(\varepsilon\geq m_\infty\). With the notation of Lemma \ref{lem:LemmeFonda}, it is natural to consider, given \(\hat{w}_0\in \hX_k\), the map 
\begin{eqnarray}\label{eq:Defhatvarphi}
\hat{\varphi}_{\hat{w}_0}:\A&\to&\left(\C^{\I}\right)^{k+1}\\
\a&\mapsto&\left(\ell^0_{\bullet}(\a,\hat{w}_0),\dots, \ell^{k}_{\bullet}(\a,\hat{w}_0)\right).\nonumber
\end{eqnarray}
This map is not canonical since it depends on the choices made during the proof of Lemma \ref{lem:LemmeFonda},  nevertheless, in view of this lemma, we will be able to use it to obtain crucial information on \(\hat{\varPhi}(\bullet,\hat{w}_0)\). The map \(\hat{\varphi}_{\hat{w}_0}\) is particularly interesting because it is linear, hence much simpler to study than \(\hat{\varPhi}(\bullet,\hat{w}_0)\). We will need to have  precise information on the rank of \(\hat{\varphi}_{\hat{w}_0}\).
\begin{lemma}\label{lem:Rank}
Same notation as above. For \(x=\pi_{0,k}\circ \nu_k(w_0)\), one has
\begin{equation}\label{eq:LemmeFonda2}
\rk \rho_x\circ\hat{\varphi}_{\hat{w}_0}=(k+1)\#\I_x.
\end{equation}

\end{lemma}

\begin{proof} Take the  notation of the proof of Lemma \ref{lem:LemmeFonda}.
Up to the isomorphism \((\C^{\I_x})^{k+1}\cong (\C^{k+1})^{\I_x}\) one can see \(\rho_x\circ\hat{\varphi}_{\hat{w}_0}\) as the map
\[\rho_x\circ \hat{\varphi}_{\hat{w}_0}=\left(\hat{\varphi}_I\right)_{I\in \I_x}\]
where for each \(I\in \I_x\), \(\hat{\varphi}_I\) is defined by 
\begin{eqnarray*}
\hat{\varphi}_I:H^0(X,A^{\q\varepsilon})&\to &\C^{k+1}\\
a_I&\mapsto& \left(\ell^0_{I}(a_I)(\hat{w}_0),\dots, \ell^k_{I}(a_I)(\hat{w}_0)\right).
\end{eqnarray*}
Observe that \(\rk(\rho_x\circ\hat{\varphi}_{\hat{w}_0})=\sum_{I\in \I_x}\rk \hat{\varphi}_I\), therefore, to prove \eqref{eq:LemmeFonda2}, it suffices to prove that for any \(I\in \I_x\), one has 
\begin{equation}\label{eq:LemmeFundaSubGoal}
\rk{\hat{\varphi}_I}=k+1.
\end{equation}
To do so, consider the family \((\tilde{b}_0,\dots, \tilde{b}_k)\) as in \eqref{eq:DefBaseb} above. Observe that from \eqref{eq:Defhatell} one infers that
\begin{eqnarray*}
\left|
\begin{array}{ccc}
\ell_I^0(\tilde{b}_0)&\cdots&\ell_I^0(\tilde{b}_k)\\
\vdots& &\vdots\\
\ell_I^k(\tilde{b}_0)&\cdots&\ell_I^k(\tilde{b}_k)
\end{array}
\right|
&=&\frac{1}{\tau_U^{r(k+1)I}\det G}
\left|
\begin{array}{ccc}
d^{[0]}_{U}(\tilde{b}_0\tau^{(r+k)I})&\cdots& d^{[0]}_{U}(\tilde{b}_k\tau^{(r+k)I})\\
\vdots & &\vdots \\
d^{[k]}_{U}(\tilde{b}_0\tau^{(r+k)I})&\cdots&d^{[k]}_{U}(\tilde{b}_k\tau^{(r+k)I})
\end{array}
\right|\\
&=&\frac{\nu_k^*\omega(\tilde{b}_0\tau^{(r+k)I},\dots,\tilde{b}_k\tau^{(r+k)I})_{U_{w_0}}}{\tau_U^{r(k+1)I}\nu_k^*\omega({b}_0,\dots, {b}_k)_{U_{w_0}}}=\frac{\tau_U^{(k+1)(r+k)I}\nu_k^*\omega(\tilde{b}_0,\dots,\tilde{b}_k)_{U_{w_0}}}{s_U^{k+1}\tau_U^{r(k+1)I}\nu_k^*\omega(\tilde{b}_0,\dots, \tilde{b}_k)_{U_{w_0}}}\\
&=&\frac{\tau_U^{k(k+1)I}}{s_U^{k+1}},
\end{eqnarray*}
(recall that  \(s(x)\neq 0\)). Since we supposed that \(\tau^I(x)\neq 0\), this determinant is non-zero when evaluated at the point \(\hat{w}_0\), this implies that \(\hat{\varphi}_I(\tilde{b}_0)\wedge\cdots\wedge \hat{\varphi}_I(\tilde{b}_k)\neq 0\), hence \(\rk\hat{\varphi}_I=k+1\),  thus proving \eqref{eq:LemmeFonda2}.
\end{proof}
From this we will be able to control the indeterminacy locus of \(\hPhi\). Let us define \(\hX_k^\circ\bydef (\pi_{0,k}\circ \nu_k)^{-1}(X^\circ)\).
\begin{proposition}\label{prop:RegularMorphism} Suppose \(N\geq n\geq 2\), \(k\geq 1\), \(\varepsilon \geq m_\infty\) and  \(\delta\geq n(k+1)\).
 Then there exists a non-empty Zariski open subset \(\A_{\rm def}\subset \A_{\rm sm}\) such that \(\hPhi|_{\A_{\rm def}\times \hX_k^\circ}\) is a (regular) morphism.
\end{proposition}
\begin{proof}
The indeterminacy locus of \(\hPhi|_{\A\times \hX_k^\circ}\) is contained in
\[
Z=
\left\{
(\a,\hat{w})\in \A\times \hX_k^\circ\midbar \hom_{I_0,\dots, I_k}(a_{I_0},\dots, a_{I_k})(\hat{w})= 0\ \forall I_0,\dots, I_k\in \I
\right\}.
\]
denoting by \(\hat{\pr}_1:\A\times \hX_k^\circ\to \A\) and \(\hat{\pr}_2:\A\times \hX_k^\circ\to \hX_k^\circ\) the two natural projections, we aim to prove that \(Z\) doesn't dominate \(\A\) via \(\hat{\pr}_1\). This will follow at once if one proves that 
\begin{equation}\label{eq:PropDefSubGoal}
\dim Z<\dim \A.
\end{equation}
Fix \(\hat{w}_0\in \hX_k^\circ\), set \(x=\pi_{0,k}\circ\nu_k(\hat{w}_0)\) and define \(Z_{\hat{w}_0}\bydef Z\cap \hat{\pr}_2^{-1}(\hat{w}_0)\). Consider the map \(\hat{\varphi}_{\hat{w}_0}\) defined by \eqref{eq:Defhatvarphi}. From Lemma \ref{lem:LemmeFonda} one sees that \(\hat{\pr}_1(Z_{\hat{w}_0})=\hat{\varphi}^{-1}_{\hat{w}_0}(\Delta),\) where
\[
\Delta
\bydef
\left\{
(v^0_{\bullet},\dots, v^k_{\bullet})\in \left(\C^{\I}\right)^{k+1}\midbar \dim \Span(v^0_{\bullet},\dots, v^k_{\bullet})<k+1
\right\}.
\]
But certainly, if one defines moreover 
\[
\Delta_x
\bydef
\left\{
(v^0_{\bullet},\dots, v^k_{\bullet})\in \left(\C^{\I_x}\right)^{k+1}\midbar \dim \Span(v^0_{\bullet},\dots, v^k_{\bullet})<k+1
\right\},
\]
one has \(\Delta\subset \rho_x^{-1}(\Delta_x)\), and therefore 
\[
\hat{\pr}_1(Z_{\hat{w}_0})\subset (\rho_x\circ\hat{\varphi}_{\hat{w}_0})^{-1}(\Delta_x).
\]
Observe that \(\dim \Delta_x=k\#\I_x+k\). Moreover, one has \(\rk(\rho_x\circ\hat{\varphi}_{\hat{w}_0})=(k+1)\#\I_x\) in view of  Lemma \ref{lem:Rank}. Therefore
\begin{eqnarray*}
\dim Z_{\hat{w}_0}&=&\dim \hat{\pr}_1(Z_{\hat{w}_0})\leq \dim (\rho_x\circ\hat{\varphi}_{\hat{w}_0})^{-1}(\Delta_x)\leq \dim\Delta_x+\dim\ker(\rho_x\circ\hat{\varphi}_{\hat{w}_0})\\
&\leq& k\#\I_x+k+\dim\A-(k+1)\#\I_x= \dim\A+k-\#\I_x.
\end{eqnarray*}
Therefore, 
\begin{eqnarray}\label{eq:EstimationDimension}
\dim Z&\leq& \dim \hX_k+\dim\A+k-\min_{x\in X^\circ}(\#\I_x)\leq n+k(n-1)+\dim \A+k-\delta-1<\dim\A
\end{eqnarray}
in view of \eqref{eq:BorneIx} and  of our hypothesis on \(\delta\). It suffices to take \(\A_{\rm def}\bydef \left(\A\setminus \hat{\pr}_1(Z)\right)\cap \A_{\rm sm}\).
\end{proof}

\begin{remark}
The hypothesis  on \(\delta\) in Proposition \ref{prop:RegularMorphism}, while sufficient for our purposes, is not optimal. As immediately follows form \eqref{eq:EstimationDimension} and \eqref{eq:BorneIx}, the same conclusion would still hold if \(\delta\) satisfies for instance
\begin{equation}\label{eq:DeltaOptimal1}
\binom{N-n+\delta}{\delta}>\dim X_k+k.
\end{equation}
\end{remark}

\subsection{Maps to families of negative dimensional complete intersection varieties} Suppose from now on that  \(N\geq n\geq 2\), \(k\geq N-1\), \(\varepsilon\geq m_{\infty}\) and that \(\delta\geq n(k+1)\) (or that \(\delta\) satisfies \eqref{eq:DeltaOptimal1}). To complete the set-up for our proof we need one more ingredient,  to construct suitable maps to families of ``negative dimensional complete intersection varieties''. To do this properly we need to consider the natural stratification on \(X\) induced by the vanishing of the \(\tau_j\)'s. The necessity of using this stratification comes from our particular choice of equation \(F(\a)\), and seems unavoidable. It was already present less explicitly  in \cite{Bro15}, then it was developed and used in a systematic way in \cite{Xie15}, and was also crucial in \cite{B-D15} and \cite{Xie16}.

For any \(J\subset \{0,\dots, N\}\) define 
\begin{eqnarray*}
X_J&\bydef& \left\{x\in X\midbar \tau_j(x)=0\Leftrightarrow j\in J\right\},\\
\I_J&\bydef& \left\{I\in \I\midbar \Supp(I)\subset \left\{0,\dots, N\right\}\setminus J\right\}.
\end{eqnarray*}
Observe that \(x\in X_J\) if and only if \(\I_x=\I_J\). Since  the \(\tau_j\)'s are in general position one obtains that 
\[\dim X_J=\max\{-1,n-\# J\},\]
where by \(\dim X_J=-1\) we mean \(X_J=\varnothing\). Therefore, \((X_J)_{\# J\leq n}\) defines  a stratification on \(X\). For any \(J\subset \{0,\dots, N\}\), let us define
\[
\P_J\bydef\left\{[T_0,\dots, T_N]\in \P^N\midbar T_j=0\ \ \text{if}\ \ j\in J\right\}.
\]
One can naturally identify \(\C^{\I_J}\bydef\bigoplus_{I\in \I_J}\C\) with \(H^0\left(\P_J,\OO_{\P_J}(\delta)\right)\cong \C\left[(T_{j'})_{j'\in \{0,\dots, N\}\setminus J} \right]_\delta\), the space of homogenous degree \(\delta\) polynomials in the variables \(T_{j'}\) with \(j'\not\in J\). This identification is realized by the map 
\[(c_I)_{I\in \I_J}\mapsto \sum_{I\in \I_J}c_IT^I.\]
For \(J=\varnothing\) this just gives the natural identification between \(\C^{\I}\) and \(H^0(\P^N,\OO_{\P^N}(\delta))\cong \C[T_0,\dots,T_N]_{\delta}\). Given \(\Delta\in \Gr_{k+1}(\C^{\I})\cong\Gr_{k+1}(\C[T_0,\dots,T_N]_{\delta})\) and \([T]\in \P^N\), write \(\Delta([T])=0\) if \(P(T)=0\) for all \(P\in \Delta\subset \C[T_0,\dots,T_N]_{\delta}\).
If \(\Delta=\Span(P_0,\dots, P_k)\), this condition  is equivalent to 
\begin{equation}\label{eq:DefDelta=0}
P_0(T)=0,\dots, P_k(T)=0.
\end{equation}
Consider the family
\[
\Y\bydef\left\{
(\Delta,[{T}])\in \Gr_{k+1}(\C^{\I})\times \P^N\midbar \Delta([{T}])=0
\right\}.
\]
Consider the map
\begin{eqnarray*}
\hat{\varPsi}:\A_{\rm def}\times \hX_{k}^\circ &\to& \Gr_{k+1}(\C^{\I})\times \P^N\\
(\a,\hat{w})&\mapsto& \Big(\hat{\varPhi}(\a,\hat{w}),[\tau^r(\hat{w})]\Big).
\end{eqnarray*}
 Where \([\tau^r(\hat{w})]\bydef\big[\tau_0^r(\pi_{0,k}\circ\nu_k(\hat{w})),\dots, \tau_N^r(\pi_{0,k}\circ\nu_k(\hat{w}))\big]\). Recall from Section \ref{sse:Setting} how we defined \(\H\subset \A_{\rm sm}\times X\) and  \(\hat{\H}^{\rm rel}_k\subset \A_{\rm sm}\times \hX_k\).  We will be interested in \(\hat{\varPsi}|_{\hat{\H}_k^{\rm rel}}\) and for this reason we will restrict ourselves to the locus where this map is regular. Let us therefore define
 \[\A_{\rm def}^\circ\bydef \A_{\rm def}\cap \big\{\a\in \A\midbar H_{\a}\cap \Sigma=\varnothing \big\}.\]
 Since \(\Sigma\) is at most a finite number of points, \(\A^\circ_{\rm def}\) is a non-empty Zariski open subset of \(\A\). Moreover, it follows form Proposition \ref{prop:RegularMorphism} that \(\hat{\varPsi}|_{\hat{\H}^{\rm rel}_{k}\cap (\A^\circ_{\rm def}\times \hat{X}_k)}\) is regular since \(\hat{\H}^{\rm rel}_{k}\cap (\A^\circ_{\rm def}\times \hat{X}_k)\subset\A_{\rm def}\times \hat{X}^{\circ}_k\).

For any \(J\subset \{0,\dots, N\}\),  set 
\[\Y_{J}\bydef \Y\cap\left(\Gr_{k+1}(\C^{\I})\times \P_J\right)\subset \Gr_{k+1}(\C^{\I})\times \P^N,\]
set also \(\hX_{k,J}\bydef \nu_{k}^{-1}(\pi_{0,k}^{-1}(X_J))\),   and let us define 
\[
\hat{\H}^{\rm rel}_{k,J}\bydef\hat{\H}^{\rm rel}_k\cap\left(\A_{\rm def}^\circ \times \hX_{k,J}\right)\subset \hat{\H}^{\rm rel}_k\cap\left(\A_{\rm def}^\circ\times \hX_k\right).
\] 
One has the following.
 \begin{proposition}\label{prop:FactoY} For any \(J\subset \{0,\dots, N\}\),
 when restricted to \(\hat{\H}^{\rm rel}_{k,J}\) the morphism \(\hat{\varPsi}\) factors through \(\Y_J\), 
 \[\hat{\varPsi}|_{\hat{\H}^{\rm rel}_{k,J}}:\hat{\H}^{\rm rel}_{k,J}\to \Y_J\subset  \Gr_{k+1}(\C^{\I})\times \P_J.\]
 \end{proposition}

\begin{proof}  It suffices to prove that \(\hat{\varPsi}\) restricted to \(\A_{\rm def}\times \hX^{\circ}_{k,J}\) factors through \(\Gr_{k+1}(\C^{\I})\times \P_J\) and that \(\hat{\varPsi}\) restricted to \(\hat{\H}^{\rm rel}_k\) factors through \(\Y\). To prove  the first statement is straightforward, therefore we now focus on proving the second one.
Since  \(\hPhi=\varPhi\circ \nu_k\), one sees that it suffices to prove that the rational map 
\begin{eqnarray*}
\varPsi:\A\times X_k&\dashrightarrow& \Gr_{k+1}(\C^{\I})\times \P^N\\
(\a,w)&\mapsto&\left(\varPhi(\a,w),[\tau^r(w)]\right)
\end{eqnarray*}
factors through \(\Y\) when restricted to \(\H^{\rm rel}_k\subset \A_{\rm sm}\times X_k\). Fix \((\a,w_0)\in \H^{\rm rel}_k\) outside the indeterminacy locus of \(\varPhi\). Take a neighborhood \(U_{w_0}\) of \(w_0\), a family \((\gamma_w)_{w\in U_{w_0}}\) as in \eqref{eq:LiftJets} and a neighborhood \(U\) of \(\pi_{0,k}(w_0)\) as in Lemma \ref{lem:localfacto}. By construction, \(\H^{\rm rel}_{k,\a}\bydef (\rho\circ \pi_{0,k})^{-1}(\a)=H_{\a,k}\), the \(k\)-th order jet space  associated to \(H_{\a}\subset X\). One obtains therefore that \([\gamma_{w_0}]_k\in J_kH_{\a}\cap p_k^{-1}(U)\), which implies that \(d_U^{[p]}F(\a)([\gamma_{w_0}]_k)=0\) for all \(0\leq p\leq k\). But by Lemma \ref{lem:localfacto},
\begin{eqnarray*}
d_U^{[p]}F(\a)=\sum_{I\in \I}d_U^{[p]}\big(a_I\tau^{(r+k)I}\big)=\sum_{I\in \I}\big(d^{[p]}_{I,U}(a_I)\big)\tau^{rI}.
\end{eqnarray*}
It then follows from the definition of \(\varPhi(\a,w_0)\), the definition of \(\Y\), \eqref{eq:DescriptionPhi} and \eqref{eq:DefDelta=0}, that \(\varPhi(\a,w_0)\in \Y\). 

 \end{proof}
As in \cite{B-D15}, the key argument in the proof of Theorem \ref{thm:nef} relies on the study of the non-finite  locus of the families \(\Y_J\). For \(J\subset \{0,\dots, N\}\), denote by \(p_J:\Y_J\to \Gr_{k+1}(\C^{\I})\)  the first projection, and define 
\begin{eqnarray*}
E_J&\bydef &
\left\{  y\in \Y\midbar \dim_y(p_J^{-1}(p_J(y)))>0
\right\}\\
\G^{\infty}_J&\bydef& p_J(E_J)\subset \Gr_{k+1}(\C^{\I}).
\end{eqnarray*}
The next lemma will be crucial for us. Let us denote, for any \(J\subset\{0,\dots, N\}\), \(\hX_{k,J}^{\circ}\bydef \hX_{k,J}\cap \hX_k^{\circ}.\)
\begin{lemma}\label{lem:AvoidingE}
For any \(J\subset \{0,\dots, N\}\). If \(\delta\geq \dim \hX_k\), then there exists a non-empty Zariski open subset \(\A_{J}\subset \A_{\rm def}\) such that 
\begin{equation}
\hat{\varPhi}^{-1}(\G_J^{\infty})\cap \big(\A_{J}\times \hX_{k,J}^{\circ}\big)=\varnothing.
\end{equation}
\end{lemma}
\begin{proof}
For \(J\subset \{0,\dots, N\}\), define moreover the following  analogues of \(\Y_J\) parametrized by affine spaces.
\begin{eqnarray*}
\widetilde{\Y}_{1,J}&\bydef&\left\{(P_0,\dots, P_k,[T])\in (\C^{\I})^{k+1}\times \P_J\midbar P_0(T)=0,\dots, P_k(T)=0\right\},\\
\widetilde{\Y}_{2,J}&\bydef&\left\{(P_0,\dots, P_k,[T])\in (\C^{\I_J})^{k+1}\times \P_J\midbar P_0(T)=0,\dots, P_k(T)=0\right\}.
\end{eqnarray*}
Where we used the identifications \(\C^{\I}\cong H^0(\P^N,\OO_{\P^N}(\delta))\) and \(\C^{\I_J}\cong H^0\big(\P_J,\OO_{\P_J}(\delta)\big)\). By analogy with \(\G_J^{\infty}\), let us denote by \(\mathbb{V}^{\infty}_{1,J}\) (resp. \(\mathbb{V}^{\infty}_{2,J}\)) the set of elements in \(\big(\C^{\I}\big)^{k+1}\) (resp. \(\big(\C^{\I_J}\big)^{k+1}\)) at which the fiber in \(\widetilde{\Y}_{1,J}\) (resp. \(\widetilde{\Y}_{2,J}\)) has a positive dimensional component.

First one checks by a straightforward computation that if one denotes by \(\rho_J:\big(\C^{\I}\big)^{k+1}\to \big(\C^{\I_J}\big)^{k+1}\)  the natural map  induced by the restriction from \(\P^N\) to \(\P_J\), one has 
\begin{eqnarray*}
\mathbb{V}_{1,J}^{\infty}= \rho_J^{-1}(\mathbb{V}_{2,J}^{\infty}),
\end{eqnarray*}
simply because for any \([T]\in \P_J\) and any \((c_I)_{I\in \I}\in \C^{\I}\), one has 
\(\sum_{I\in \I}c_IT^I=\sum_{I\in \I_J}c_IT^I\).

Moreover, by a result due to Benoist \cite{Ben11} (see \cite{B-D15}), one has 
\begin{equation}\label{eq:Benoist}
\codim_{(\C^{\I_J})^{k+1}}\mathbb{V}_{2,J}^{\infty}\geq \delta+1.
\end{equation}
We are now going to bound the dimension of \(\hat{\varPhi}^{-1}(\G_J^{\infty})\cap \left(\A_{\rm def}\times \hX_{k,J}^{\circ}\right)\). Take \(\hat{w}_0\in \hX_{k,J}^{\circ}\) and take \(\hat{\varphi}_{\hat{w}_0}\) as in \eqref{eq:Defhatvarphi}. From Lemma \ref{lem:LemmeFonda} one obtains that 
\[
\hat{\varPhi}^{-1}(\G_J^{\infty})\cap (\A_{\rm def}\times \{\hat{w}_0\})\cong \hat{\varphi}_{\hat{w}_0}^{-1}(\mathbb{V}_{1,J}^{\infty})\cap \A_{\rm def}= (\rho_J\circ \hat{\varphi}_{\hat{w}_0})^{-1}(\mathbb{V}_{2,J}^{\infty})\cap \A_{\rm def}.
\]
But since \(x\bydef \pi_{0,k}(\nu_k(\hat{w}_0)) \in X_J\), we have \(\I_x=\I_J\), hence \(\rho_x=\rho_J\). Lemma \ref{lem:Rank} thus implies that 
\[\rk(\rho_J\circ \hat{\varphi}_{\hat{w}_0})=(k+1)\#\I_J=\dim (\C^{\I_J})^{k+1}.\]
Therefore
\begin{eqnarray*}
\dim\left(\hPhi^{-1}(\G^{\infty}_J)\cap (\A_{\rm def}\times \{\hat{w}_0\})\right)&\leq& \dim (\rho_J\circ \hat{\varphi}_{\hat{w}_0})^{-1}(\mathbb{V}^{\infty}_{2,J})\leq \dim \mathbb{V}^{\infty}_{2,J}+\dim \ker(\rho_J\circ \hat{\varphi}_{\hat{w}_0})\\
&\leq& \dim(\C^{\I_J})^{k+1}-\codim _{(\C^{\I_J})^{k+1}}\mathbb{V}_{2,J}^{\infty} +\dim \A-\rk(\rho_J\circ \hat{\varphi}_{\hat{w}_0})\\
&=&\dim\A-\codim _{(\C^{\I_J})^{k+1}}\mathbb{V}_{2,J}^{\infty}.
\end{eqnarray*}
A final computation then yields
\begin{eqnarray*}
\dim\left(\hPhi^{-1}(\G^{\infty}_J)\cap \big(\A_{\rm def}\times \hX_{k,J}\big)\right)\leq \dim \A-\codim _{(\C^{\I_J})^{k+1}}\mathbb{V}_{2,J}^{\infty}+\dim \hX_{k,J}<\dim \A
\end{eqnarray*}
in view of \eqref{eq:Benoist} and our hypothesis on \(\delta\). It then suffices to set 
\(\A_J\bydef\A_{\rm def}\setminus \pr_1\left(\hPhi^{-1}(\G^{\infty}_J)\cap\big( \A_{\rm def}\times \hX_{k,J}\big)\right).\)
\end{proof}
\begin{remark} Observe that this proof shows that the conclusion of   Lemma \ref{lem:AvoidingE}  would still hold if the condition on \(\delta\)  is replaced by the  condition 
\begin{equation}\label{eq:DeltaOptimal2}\codim _{(\C^{\I_J})^{k+1}}\mathbb{V}_{2,J}^{\infty}> \dim \hX_{k,J}.\end{equation}
\end{remark}
\subsection{Proof of Theorem \ref{thm:nef}}\label{sse:proof} We are now in position to prove Theorem \ref{thm:nef}. Take \(N\geq n\), \(k\geq N-1\), \(\varepsilon \geq m_{\infty}\) and \(\delta\geq n(k+1)\geq n+k(n-1)=\dim \hX_k\) (or such that \(\delta\) satisfies \eqref{eq:DeltaOptimal1} and \eqref{eq:DeltaOptimal2} for any \(J\)). Let us denote by \(\Q\) the (very ample) Pl\"ucker line bundle on \(\Gr_{k+1}(\C^{\I})\). Let us also denote, for any \(J\subset\{0,\dots, N\}\), by \(q_1\) and \(q_{2}\) the canonical projections from \(\Gr_{k+1}(\C^{\I})\times \P_J\) to each factors, the ambiguity of the notation for \(q_2\) should not lead to any confusion. By the definition of \(\hat{\varPsi}\) one obtains that for any \(m\in \N\), 
\begin{equation}\label{eq:PullBack}
\hat{\varPsi}^*\big(q_1^*\Q^m\otimes q_2^*\OO_{\P^N}(-1)\big)
=\nu_k^{*}\left(\OO_{X_k}(mk')\otimes \pi_{0,k}^*A^{m(k+1)(\q\varepsilon+k\p\delta)-\p r}\right)\otimes\OO_{\hX_k}(-mF). 
\end{equation}
Here we took \(q_2\) for \(J=\varnothing\). The key point in this formula is the isolated \(-\p r\).

For any \(J\subset \{0,\dots, N\}\), by Nakamaye's theorem on the augmented base locus \cite{Nak00}, and the definition of \(E_J\), one obtains that \(E_J\) is precisely the augmented base locus \(\B_+(q_1^*\Q|_{\Y_J})\) of \(q_1^*\Q|_{\Y_J}\). Since \(q_1^*\Q\otimes q_2^*\OO_{\P_J}(1)\) is very ample, one obtains from the definition of \(\B_+\), by noetherianity, that there exists \(m_J\in \N\) such that 
\begin{equation}\label{eq:B+}
E_J=\B_+(q_1^*{\Q|_{\Y_J}})=\Bs\left(q_1^*\Q^m\otimes q_2^*\OO_{\P_J}(-1)|_{\Y_J}\right), \ \ \ \forall \ m\geq m_J.
\end{equation}
Set 
\(M\bydef \max\left\{m_J\midbar J\subset\{0,\dots, N\}\right\}\), observe that \(M\) only depends on \(N,k,\delta\), 
and define
\[r(\p,\q,M,N,k,\varepsilon, \delta)\bydef \left\lceil \frac{M(k+1)(\q\varepsilon+k\p\delta)+1}{\p}\right\rceil\ \ \ \text{and}\ \ \ \A_{\rm nef}\bydef \bigcap_{J\subset \{0,\dots, N\}}\A_J\cap \A_{\rm def}^{\circ}.\]
Let us prove that the conclusion of Theorem \ref{thm:nef} is then satisfied. Take \(\a\in \A_{\rm nef}\). We aim to prove that 
\[\nu_{k}^*\left(\OO_{X_k}(Mk')\otimes \pi_{0,k}^*A^{-1}\right)\otimes \OO_{\hX_k}(-MF)|_{\hat{H}_{k,\a}}\]
is nef on \(\hat{H}_{k,\a}\subset \hX_k\). Take an irreducible curve \(C\subset \hat{H}_{k,\a}\) and take (the unique) \(J\subset \{0,\dots, N\}\) such that \(\hX_{k,J}\cap C=C^{\circ}\) is a non-empty open subset of \(C\). Therefore \(C^{\circ}\subset \hat{\H}^{\rm rel}_{k,J}\), and by Proposition \ref{prop:FactoY}, \(\hat{\varPsi}|_{C^\circ}\) factors through \(\Y_J\), and since \(\Y_J\) is proper, \(\hat{\varPsi}|_C\) factors through \(\Y_J\) as well. But from Lemma \ref{lem:AvoidingE} one obtains that \(\hPhi(C^\circ)\cap \G_J^{\infty}=\varnothing\) and that therefore \(\hat{\varPsi}(C^\circ)\cap E_J=\varnothing\) so that in particular
\[\hat{\varPsi}(C)\not\subset E_J.\]
From this, and \eqref{eq:B+}, it follows that \(\hat{\varPsi}(C)\cdot\left( q_1^*\Q^M\otimes q_2^*\OO_{\P_J}(-1)\right)\geq 0\) and that therefore 
\[
C\cdot \hat{\varPsi}^*\left(q_1^*\Q^M\otimes q_2^*\OO_{\P^N}(-1)\right)\geq 0.
\]
Combining this equality with our hypothesis \(r\geq r(\p,\q,M,N,k,\varepsilon,\delta)\), \eqref{eq:PullBack} and the fact that \(\nu_k^*\pi_{0,k}^*A\) is nef, it follows that 
\[
C\cdot\left( \nu_{k}^*\left(\OO_{X_k}(Mk')\otimes \pi_{0,k}^*A^{-1}\right)\otimes \OO_{\hX_k}(-MF)\right)\geq 0.
\]
This proves the desired nefness and concludes the proof of Theorem \ref{thm:nef}.\\

\emph{Acknowledgment:} We warmly thank  Lionel Darondeau for his support, his help and the conversations we had. We gratefully thank Olivier Benoist for his help and for the insightful  suggestions he provided. We thank Simone Diverio for his useful  comments on the present work, as well as Carlo Gasbarri for the discussions we had and for the simplifications he pointed out. We also thank Ya Deng for the improvements he suggested. 
\bibliographystyle{plain}

\end{document}